\documentclass[11pt]{amsart}
\usepackage{geometry}                
\usepackage{graphicx}
\usepackage{amssymb}
\usepackage{epstopdf}
\newtheorem{theorem}{Theorem}[section]
\newtheorem{definition}[theorem]{Definition}

\newtheorem{corollary}[theorem]{Corollary}
\newtheorem{proposition}[theorem]{Proposition}
\newtheorem{lemma}[theorem]{Lemma}

\title[Non-arithmetic monodromy of higher hypergeometric functions]
{Non-arithmetic monodromy\\of higher hypergeometric functions}

\author{John R.~Parker}
\date{\today}                                           

\begin{document}

\begin{abstract}
We show that all the currently known non-arithmetic lattices in ${\rm PU}(2,1)$
are monodromy groups of higher hypergeometric functions.
\end{abstract}

\maketitle

\section{Introduction}

\noindent
A classical result of Schwarz \cite{HAS} says that any (orientation
preserving) triangle group is the monodromy a hypergeometric differential
equation. That is, it is the group of linear maps that record how
pairs of hypergeometric functions solving this equation vary when
analytically continued around a singularity of the equation.
This generalises the well known fact that the modular group 
${\rm PSL}(2,{\mathbb Z})$ is the monodromy group of an elliptic function.
The latter connection raises questions about the arithmetic nature of
such monodromy groups. Vinberg \cite{Vinberg} and Takeuchi \cite{KT-1} 
gave criteria for the arithmeticity of reflection groups and 
Fuchsian groups, respectively. Subsequently, 
Takeuchi \cite{KT} showed that all but finitely
many triangle groups in ${\rm PSL}(2,{\mathbb R})$ are non-arithmetic.
Combining the results of Schwarz and Takeuchi gives infinitely many
non-arithmetic hypergeometric monodromy groups.

We now discuss two generalisations of Schwarz's result about hypergeometric
monodromy.
First, in \cite{DM} and \cite{Mostow-IHES} Deligne and Mostow considered the
monodromy of hypergeometric functions in $n$ variables, originally
constructed by Picard \cite{Picard}. These 
monodromy groups live in ${\rm PU}(n,1)$ (the case $n=1$ gives
the classical case since ${\rm PU}(1,1)$ is conjugate to 
${\rm PSL}(2,{\mathbb R})$). Mostow \cite{Mostow-RCP} generalised
Vinberg's arithmeticity criterion to ${\rm PU}(n,1)$ and gave the first examples
of non-arithmetic lattices in ${\rm PU}(2,1)$. Deligne and Mostow
showed that Mostow's lattices are monodromy groups of second order 
hypergeometric equations in $2$ variables and they produced other examples. In 
particular, for many years all the known examples of non-arithmetic
lattices in ${\rm PU}(2,1)$ and ${\rm PU}(3,1)$ were contained in the
Deligne-Mostow list, and hence were monodromy groups for hypergeometric
functions. (It is an open question whether ${\rm PU}(n,1)$ contains
non-arithmetic lattices for $n\ge 4$.)

In \cite{DPP-1} and \cite{DPP} Deraux, Paupert and I gave some new examples 
of non-arithmetic lattices in ${\rm PU}(2,1)$, which are not 
commensurable to groups on the Deligne-Mostow list. 
(Recall, that two groups are said to be commensurable if, after 
conjugating one of them if necessary, their intersection has finite index 
in each of them. This is the natural notion of invariance for
non-arithmetic lattices.) When discussing these examples, a
question we have frequently been asked is whether
the new non-arithmetic lattices constructed in \cite{DPP-1} and \cite{DPP} 
arise as monodromy groups for any functions or differential equations.
In this paper we give an answer to this question.
For completeness, we briefly mention that 
Deraux \cite{Deraux-abelian}, \cite{Deraux-CHL} has shown that
some of the new lattices we construct in \cite{DPP-1} and \cite{DPP} 
had already been found by 
Couwenberg, Heckman and Looijenga \cite{CHL}, who do not discuss their
arithmetic properties. 

A second generalisation of hypergeometric monodromy was studied in detail 
by Beukers and Heckman \cite{BH}. They consider higher order hypergeometric 
equations in one variable, first constructed by Thomae \cite{Thomae}. 
In particular, they give a characterisation 
of hypergeometric groups due to Levelt \cite{Le} and also
a method of calculating the signature of the Hermitian form 
preserved by such a group. We discuss this in 
Section~\ref{sec-higher-monodromy} below. Arithmetic monodromy groups of
these higher hypergeometric equations have been studied by
Singh and Venkataramana \cite{SV} and by Fuchs, Mieri and Sarnak
\cite{FMS}. 

Our main result is:

\begin{theorem}\label{thm-main}
With the exception of ${\mathcal T}(p,{\bf E}_2)$ for $p=3,\,6,\,12$, 
all the lattices in ${\rm PU}(2,1)$ constructed by Deraux, Parker and
Paupert in \cite{DPP} are commensurable to monodromy groups of third order
hypergeometric equations. 
\end{theorem}

An immediate consequence is:

\begin{corollary}\label{cor-main}
Each currently known commensurability
class of non-arithmetic lattices in ${\rm PU}(2,1)$ contains a 
monodromy group for a third order hypergeometric equation.
\end{corollary}

We remark that a consequence is that the Deligne-Mostow 
lattices in ${\rm PU}(2,1)$ are commensurable both to monodromy groups 
of second order hypergeometric equations in two variables and to monodromy 
groups of third order hypergeometric equations in one variable. It is 
not clear whether, for each group, there is any relationship between these 
equations. 

It would be interesting to know whether there are any additional higher 
hypergeometric equations with non-arithmetic monodromy, and perhaps
this could give a place to start looking for more non-arithmetic lattices.

In Section~\ref{sec-higher-monodromy} we review the necessary background
on higher order hypergeometric equations and their solutions, following 
Beukers and Heckman \cite{BH}. 
In Section~\ref{sec-triangle-groups} we review groups generated by three
complex reflections. Our main reference is Deraux, Parker and Paupert 
\cite{DPP}. In Section~\ref{sec-main-proof} we combine the previous two
sections in order to prove Theorem \ref{thm-main}. The proof is split 
into different cases according to the families considered in \cite{DPP}. 
In each case we exhibit a set of generators for the monodromy group 
satisfying Levelt's criterion and we give the possible
values of the angle parameters $\alpha_j$, $\beta_j$ of the associated
higher hypergeometric equation. These results are 
Propositions~\ref{prop-levelt-3fold}, 
\ref{prop-levelt-2fold}, \ref{prop-levelt-334}, \ref{prop-levelt-23n}
and \ref{prop-levelt-344} respectively. 

\medskip

\noindent
{\bf Acknowledgements.} 
I am very grateful to the following people for their help with this project.
This problem was suggested to me by T.N.~Venkataramana, who also told me
of the paper \cite{BH} by Beukers and Heckman. I also have benefited from 
many useful conversations with Martin Deraux and Julien Paupert.

\section{Hypergeometric functions}\label{sec-higher-monodromy}

We review hypergeometric equations and functions together
with the associated monodromy groups. We do not discuss the case of
hypergeometric functions in several variables. This material is discussed
at length by Deligne and Mostow \cite{DM}, \cite{Mostow-IHES}.

\subsection{The classical case}

We begin with a brief review the classical hypergeometric equation
and hypergeometric functions; see Chapter XIV of
Whittaker and Watson \cite{WW}.

We write the Pochhammer symbol
\begin{equation}\label{eq-poch}
(\alpha)_k=\alpha(\alpha+1)\cdots(\alpha+k-1)
=\frac{\Gamma(\alpha+k)}{\Gamma(\alpha)}. 
\end{equation}
The classical hypergeometric function $_2F_1$ is defined by the series
$$
_2F_1(\alpha,\beta,\gamma;z)
=\sum_{k=0}^\infty \frac{(\alpha)_k(\beta)_k}{(\gamma)_k}\cdot\frac{z^k}{k!}.
$$
It is a solution to the hypergeometric equation
$$
z(1-z)\frac{d^2w}{dz^2}
+\bigl(\gamma-(\alpha+\beta+1)z\bigr)\frac{dw}{dz}-\alpha\beta w=0.
$$
It has three singular points at $0$, $1$ and $\infty$. 
Analytically continuing a pair of independent solutions to this
equation along a closed path around these singular points yields
two new solutions which are linear combinations of the two initial solutions.
The resulting $2\times 2$ matrix is the monodromy of the equation
associated to this path. The monodromy group was investigated by 
Schwarz \cite{HAS}. In particular, if
$1-\gamma$, $\alpha-\beta$ and $\gamma-\alpha-\beta$
are rational numbers with denominators $p$, $q$, $r$, all at least 2, then the 
monodromy group is the (orientation preserving) $(p,q,r)$ triangle group,
which is discrete. This group preserves a Hermitian form which is positive
definite when $1/p+1/q+1/r>1$, degenerate when $1/p+1/q+1/r=0$
and indefinite when $1/p+1/q+1/r<1$.

\subsection{Higher hypergeometric equations and functions}

This section is a review of higher order hypergeometric equations and
functions and it closely follows the paper \cite{BH} by Beukers and Heckman.
Our aim is to include the necessary background for later sections.
For a fuller account readers should look at \cite{BH}.

The higher hypergeometric function $_nF_{n-1}$ (see equation (1.3) of 
Beukers and Heckman \cite{BH}) is 
defined to be:
$$
_nF_{n-1}(\alpha_1,\ldots,\alpha_n;\beta_1,\ldots,\beta_{n-1};z)
=\sum_{k=0}^\infty 
\frac{(\alpha_1)_k\cdots(\alpha_n)_k}{(\beta_1)_k\cdots(\beta_{n-1})_k}
\cdot \frac{z^k}{k!}
$$
where once again $(\alpha)_k$ denotes the Pochhammer symbol \eqref{eq-poch}.
We use the differential operator $\theta=z\frac{d}{dz}$ and define
\begin{eqnarray*}
D(\alpha;\beta) & = & D(\alpha_1,\ldots,\alpha_n;\beta_1,\ldots,\beta_n) \\
& = & (\theta+\beta_1-1)\cdots(\theta+\beta_n-1)
-z(\theta+\alpha_1)\cdots(\theta+\alpha_n).
\end{eqnarray*}
Following equation (2.5) of \cite{BH}, we 
write the higher hypergeometric equation as
\begin{equation}\label{eq-HG-eq}
D(\alpha;\beta)w=D(\alpha_1,\ldots,\alpha_n;\beta_1,\ldots,\beta_n)w=0.
\end{equation}
If no pair of $\beta_1,\ldots,\,\beta_n$ differ by an integer, then
$n$ independent solutions of the equation \eqref{eq-HG-eq} are given,  
for $i=1,\,\ldots,\,n$, by
$$
z^{1-\beta_i}\,_nF_{n-1}(1+\alpha_1-\beta_i,\ldots,1+\alpha_n-\beta_i;
1+\beta_1-\beta_i,{\buildrel{\vee}\over{\cdots}},1+\beta-\beta_n-\beta_i;z)
$$
where $\vee$ denotes the variable $1+\beta_i-\beta_i$
has been omitted; equation (2.9) of \cite{BH}.

\medskip

Beukers and Heckman give the following definition of a hypergeometric group.

\begin{definition}(Definition 3.1 of Beukers and Heckman \cite{BH})
Suppose that $a_1,\,\ldots,\,a_n$; $b_1,\,\ldots,\,b_n\in {\mathbb C}-\{0\}$
with $a_j\neq b_k$ for all $j,\,k\in\{1,\,\ldots,\,n\}$. 
A {\bf hypergeometric group} 
$H(a;b)=H(a_1,\ldots,a_n;b_1,\ldots,b_n)$ 
with numerator parameters $a_1,\,\ldots,\,a_n$
and denominator parameters $b_1,\,\ldots,\,b_n$ is a subgroup of 
${\rm GL}(n,{\mathbb C})$ generated by $A$ and $B$ which have
characteristic polynomials
$$
\chi_A(t)={\rm det}(tI-A)=\prod_{j=1}^n(t-a_j),\quad
\chi_B(t)={\rm det}(tI-B)=\prod_{j=1}^n(t-b_j)
$$
and so that $BA^{-1}$ is a complex reflection, that is $(BA^{-1}-I)$ has rank 
one.

A {\bf scalar shift} of the hypergeometric group $H(a;b)$ is a hypergeometric
group $H(da;db)=H(da_1,\ldots,da_n,db_1,\ldots,db_n)$ for some
$d\in{\mathbb C}-\{0\}$.
\end{definition}

Hypergeometric groups are monodromy groups of higher hypergeometric equations:

\begin{proposition}[Proposition 3.2 of Beukers and Heckman \cite{BH}]
Suppose $a_1,\,\ldots,\,a_n$; $b_1,\,\ldots,\,b_n\in {\mathbb C}-\{0\}$
with $a_j\neq b_k$ for all $j,\,k\in\{1,\,\ldots,\,n\}$. 
Let $\alpha_1,\,\ldots,\,\alpha_n,\beta_1,\,\ldots,\,\beta_n$ be 
complex numbers satisfying $a_j=e^{2\pi i\alpha_j}$ and $b_j=e^{2\pi i\beta_j}$ for
$j=1,\,\ldots,\,n$. The monodromy group of the hypergeometric equation
$D(\alpha_1,\ldots,\alpha_n;\beta_1,\ldots,\beta_n)w=0$ is a hypergeometric 
group with parameters 
$a_1,\,\ldots,\,a_n,\,b_1,\,\ldots,\,b_n$.
\end{proposition}

Hypergeometric groups were characterised by Levelt:

\begin{theorem}[Theorem 1.1 of Levelt \cite{Le};
Theorem 3.5 of Beukers and Heckman \cite{BH}].\label{thm-levelt}
Suppose $a_1,\,\ldots,\,a_n$; $b_1,\,\ldots,\,b_n\in {\mathbb C}-\{0\}$
with $a_j\neq b_k$ for all $j,\,k\in\{1,\,\ldots,\,n\}$. 
Let $A_1,\,\ldots,\,A_n,\,B_1,\,\ldots,\,B_n$ be defined by
$$
\prod_{j=1}^n(t-a_j)=t^n+A_1t^{n-1}+\cdots+A_n,\quad
\prod_{j=1}^n(t-b_j)=t^n+B_1t^{n-1}+\cdots+B_n.
$$
Let $A$ and $B$ in ${\rm GL}(n,{\mathbb C})$ be defined by:
$$
A=\left(\begin{matrix}
0 & 0 & \cdots & 0 & -A_n \\ 
1 & 0 & \cdots & 0 & -A_{n-1} \\ 
0 & 1 & \cdots & 0 & -A_{n-2} \\
\vdots & \vdots & & \vdots & \vdots \\
0 & 0 & \cdots & 1 & -A_1 
\end{matrix}\right),\quad
B=\left(\begin{matrix}
0 & 0 & \cdots & 0 & -B_n \\ 
1 & 0 & \cdots & 0 & -B_{n-1} \\ 
0 & 1 & \cdots & 0 & -B_{n-2} \\
\vdots & \vdots & & \vdots & \vdots \\
0 & 0 & \cdots & 1 & -B_1 
\end{matrix}\right).
$$
Then the matrices $A$ and $B$ generate a 
hypergeometric group $H(a_1,\ldots,a_n,b_1,\ldots,b_n)$.
Moreover, any hypergeometric group with the same parameters is conjugate
inside ${\rm GL}(n,{\mathbb C})$ to this one. Also any scalar shift
$H(da_1,\ldots,da_n,db_1,\ldots,db_n)$ for $d\in{\mathbb C}-\{0\}$
of this group can be obtained by multiplying $A$ and $B$ by $d$ with
$|d|=1$.
\end{theorem}

In fact, it will be more convenient for us to use a normalisation
where the $-A_j$ (respectively $-B_j$) occur in the first row of $A$ 
(respectively $B$) rather than the last column. To be
precise, we require the following normal forms:
\begin{equation}\label{eq-AB-general}
A=\left(\begin{matrix}
-A_1 & -A_2 & \cdots & -A_{n-1} & -A_n \\
1 & 0 & \cdots & 0 & 0 \\
0 & 1 & \cdots & 0 & 0 \\
\vdots & \vdots & & \vdots & \vdots \\
0 & 0 & \cdots & 1 & 0 
\end{matrix}\right), \quad
B=\left(\begin{matrix}
-B_1 & -B_2 & \cdots & -B_{n-1} & -B_n \\
1 & 0 & \cdots & 0 & 0 \\
0 & 1 & \cdots & 0 & 0 \\
\vdots & \vdots & & \vdots & \vdots \\
0 & 0 & \cdots & 1 & 0 
\end{matrix}\right)
\end{equation}
It is clear that we can
go from one form to the other by taking the transpose then performing 
the same permutation of rows and columns in both matrices.

\subsection{Monodromy when $n=3$}

For the rest of the paper we restrict our attention to the case
of $3\times 3$ matrices.
We begin with matrices $A$ and $B$ in the previous section with $n=3$.
\begin{eqnarray}
A & = & \left(\begin{matrix}
-A_1 & -A_2 & -A_3 \\ 1 & 0 & 0 \\ 0 & 1 & 0 \end{matrix}\right)
\ =\ \left(\begin{matrix} 
a_1+a_2+a_3 & -a_1a_2-a_2a_3-a_3a_1 & a_1a_2a_3 \\
1 & 0 & 0 \\ 0 & 1 & 0 \end{matrix}\right),\label{eq-A-general}\\
B & = & \left(\begin{matrix} 
-B_1 & -B_2 & -B_3 \\ 1 & 0 & 0 \\ 0 & 1 & 0 \end{matrix}\right)
\ =\ \left(\begin{matrix} 
b_1+b_2+b_3 & -b_1b_2-b_2b_3-b_3b_1 & b_1b_2b_3 \\
1 & 0 & 0 \\ 0 & 1 & 0 \end{matrix}\right).\label{eq-B-general}
\end{eqnarray}
Also,
$$
A^{-1}=\left(\begin{matrix} 0 & 1 & 0 \\ 0 & 0 & 1 \\
-1/A_3 & -A_1/A_3 & -A_2/A_3 \end{matrix}\right),\quad 
B^{-1}=\left(\begin{matrix} 0 & 1 & 0 \\ 0 & 0 & 1 \\
-1/B_3 & -B_1/B_3 & -B_2/B_3 \end{matrix}\right).
$$ 
Thus
$$
BA^{-1} = \left(\begin{matrix}
B_3/A_3 & (B_3A_1-B_1A_3)/A_3 & (B_3A_2-B_2A_3)/A_3 \\
0 & 1 & 0 \\ 0 & 0 & 1 \end{matrix}\right).
$$ 
It is clear that $(BA^{-1}-I)$ has rank one, and so it is a complex
reflection.

For $j=1,\,2,\,3$ it is easy to see that the vector ${\bf a}_j$
below spans the $a_j$-eigenspace of $A$. We also write down the image
of ${\bf a}_j$ under $B$.
\begin{equation}\label{eq-aj-Baj}
{\bf a}_j=\left(\begin{matrix} a_j^2 \\ a_j \\ 1 \end{matrix}\right),\quad  
B{\bf a}_j=\left(\begin{matrix}
-B_1a_j^2-B_2a_j-B_3 \\ a_j^2 \\ a_j\end{matrix}\right)
=\left(\begin{matrix} a_j^3-\chi_B(a_j) \\
a_j^2 \\ a_j\end{matrix}\right).
\end{equation}
Let $U$ be the following matrix whose columns are eigenvectors for $A$:
$$
U=\left(\begin{matrix}
a_1^2/(a_1-a_2)(a_1-a_3) & a_2^2/(a_2-a_1)(a_2-a_3) & a_3^2/(a_3-a_1)(a_3-a_2) 
\\
a_1/(a_1-a_2)(a_1-a_3) & a_2/(a_2-a_1)(a_2-a_3) & a_3/(a_3-a_1)(a_3-a_2) \\
1/(a_1-a_2)(a_1-a_3) & 1/(a_2-a_1)(a_2-a_3) & 1/(a_3-a_1)(a_3-a_2)
\end{matrix}\right).
$$
Then
$$
U^{-1}=\left(\begin{matrix}
1 & -a_2-a_3 & a_2a_3 \\
1 & -a_3-a_1 & a_1a_3 \\
1 & -a_1-a_2 & a_1a_2 \end{matrix}\right).
$$
Using equation \eqref{eq-aj-Baj} we have:
\begin{equation}\label{eq-AB-3}
U^{-1}AU=\left(\begin{matrix} 
a_1 & 0 & 0 \\ 0 & a_2 & 0 \\ 0 & 0 & a_3 \end{matrix}\right), \quad
U^{-1}BU=\left(\begin{matrix} 
a_1(1+c_1) & a_2c_2 & a_3c_3 \\
a_1c_1 & a_2(1+c_2) & a_3c_3 \\
a_1c_1 & a_2c_2 & a_3(1+c_3)
\end{matrix}\right)
\end{equation}
where for $j,\,k\in\{1,\,2,\,3\}$:
\begin{equation}\label{eq-c}
c_j=\frac{-\chi_B(a_j)}{a_j\prod_{k\neq j}(a_k-a_j)}
=\frac{(b_j-a_j)}{a_j}\prod_{k\neq j}\frac{(b_k-a_j)}{(a_k-a_j)}
.
\end{equation}
Note that ${\rm tr}(BA^{-1})=2+(b_1b_2b_3)/(a_1a_2a_3)$ and 
${\rm tr}(B'{A'}^{-1})=3+c_1+c_2+c_3$. Since these matrices are conjugate
we have the identity:
\begin{equation}\label{eq-c-eq}
b_1b_2b_3=a_1a_2a_3(c_1+c_2+c_3+1).
\end{equation}

Beukers and Heckman show that when the eigenvalues of $A$ and $B$ satisfy
$|a_j|=|b_j|=1$, for $j=1,\,2,\,3$, then $A$ and $B$ preserve a Hermitian form
and they give give a recipe for finding the signature of of this form. We 
give a direct proof of their result in the $3\times 3$ case.

\begin{theorem}[Beukers and Heckman, Theorem~4.5 of \cite{BH}]\label{thm-bh}
Suppose that the eigenvalues $a_1$, $a_2$, $a_3$; $b_1$, $b_2$, $b_3$ 
of $A$ and $B$ are complex numbers with $|a_j|=|b_j|=1$ for $j=1,\,2,\,3$. 
Suppose also that the $a_j$ are distinct.
Let $A'=U^{-1}AU$ and $B'=U^{-1}BU$ be given by \eqref{eq-AB-3}. 
For $j,\,k\in\{1,2,3\}$ write
\begin{eqnarray*}
d_j & = & -i\left(\frac{b_1b_2b_3}{a_1a_2a_3}\right)^{-1/2}
\frac{(b_j-a_j)}{a_j}\prod_{k\neq j}\frac{(b_k-a_j)}{(a_k-a_j)}
=-i\bigl({\rm det}(BA^{-1})\bigr)^{-1/2}c_j.
\end{eqnarray*}
Then $A'$ and $B'$ preserve a diagonal Hermitian form 
$D={\rm diag}(d_1,d_2,d_3)$. Furthermore, any Hermitian form preserved
by $A'$ and $B'$ is a real multiple of $D$.
\end{theorem}

\begin{proof}
Beukers and Heckman give a proof valid for all $n\times n$ matrices
of the given form. We give a short direct proof for $3\times 3$ matrices.
It is possible to extend this proof to the $n\times n$ case, but we will
not need that below.

Let $c_j$ be given by \eqref{eq-c}. 
Observe that, since $|a_i|=|b_i|=1$ for all $i$, we have
$$
\frac{c_j}{\bar{c}_j}
=\frac{(b_j-a_j)\bar{a}_j}{a_j(\bar{b}_j-\bar{a}_j)}
\prod_{k\neq j}\frac{(b_k-a_j)(\bar{a}_k-\bar{a}_j)}
{(a_k-a_j)(\bar{b}_k-\bar{a}_j)}
=-\frac{b_1b_2b_3}{a_1a_2a_3}.
$$
Define $\psi\in[0,\pi)$ by 
\begin{equation}\label{eq-psi}
e^{i\psi}=\left(\frac{b_1b_2b_3}{a_1a_2a_3}\right)^{1/2}
=\bigl({\rm det}(BA^{-1})\bigr)^{1/2}.
\end{equation}
Thus, for $j=1,\,2,\,3$, we have $c_j/\bar{c}_j=-e^{2i\psi}$ and so 
$c_j=ie^{i\psi}d_j$ for some real number $d_j$.
Also note that, using \eqref{eq-c-eq}, we have
\begin{equation}\label{eq-d}
d_1+d_2+d_3=-ie^{-i\psi}(c_1+c_2+c_3)
=-ie^{-i\psi}\left(\frac{b_1b_2b_3}{a_1a_2a_3}-1\right)
=2\sin(\psi).
\end{equation}

Let $D$ be the diagonal matrix $D={\rm diag}(d_1,d_2,d_3)$. We claim
that the matrices $A'=U^{-1}AU$ and $B'=U^{-1}BU$ given by \eqref{eq-AB-3}
satisfy ${A'}^*DA'=D$ and ${B'}^*DB'=D$. The first of these is obvious
since $A'=U^{-1}AU$ is diagonal with entries $a_j$ where $|a_j|=1$.
A short calculation shows that the $ij$th entry of ${B'}^*DB'-D$ is
$$
\bar{a}_ia_j\bigl(\bar{c}_ic_j(d_1+d_2+d_3)+\bar{c}_id_j+c_jd_i\bigr).
$$
Using the definition of $d_j$ and equation \eqref{eq-d} we have
$$
\bar{c}_ic_j(d_1+d_2+d_3)+\bar{c}_id_j+c_jd_i
=d_id_j2\sin(\psi)+ie^{i\psi}d_id_j-ie^{-i\psi}d_id_j=0.
$$
Hence ${B'}^*DB'=D$ as claimed.

Finally, we show that any Hermitian form preserved by $A'$ and $B'$
is a real multiple of $D$. First, observe that since $A'$ is
diagonal with distinct entries then any Hermitian form it preserves
must be diagonal, say $D'={\rm diag}(d_1',d_2',d_3')$. Now consider
${B'}^*D'B'-D'$. Arguing as
above, we see that for $i,\,j\in\{1,\,2,\,3\}$ we must have
$$
\bar{c}_ic_j(d'_1+d'_2+d'_3)+\bar{c}_id'_j+c_jd'_i=0.
$$
Therefore for all $i,\,j,\,k$ we have
\begin{eqnarray*}
0 & = & c_k\bigl(\bar{c}_ic_j(d'_1+d'_2+d'_3)+\bar{c}_id'_j+c_jd'_i\bigr)
-c_j\bigl(\bar{c}_ic_k(d'_1+d'_2+d'_3)+\bar{c}_id'_k+c_kd'_i\bigr) \\
& = & \bar{c}_i(c_kd'_j-c_jd'_k).
\end{eqnarray*}
Thus there is $\lambda\in{\mathbb C}$ with 
$\lambda d'_j=c_j$ for $j=1,\,2,\,3$.
Since $d'_j$ must be real, we see that $\lambda=ie^{i\psi}$, giving
the result. 
\end{proof}

\medskip

We can give the $d_j$ in terms of angle parameters as follows. The
formula below appears on page 335 of Beukers and Heckman \cite{BH}
where it is written $F(u_j,u_j)$.

\begin{corollary}\label{cor-bh}
Let $A$ and $B$ be as in Theorem \ref{thm-bh}. If $a_j=e^{2\pi i \alpha_j}$ and 
$b_j=e^{2\pi i\beta_j}$ for $j=1,\,2,\,3$. Then the Hermitian form
$D={\rm diag}(d_1,d_2,d_3)$ preserved by $A'$ and $B'$ has entries:
$$
d_j = 2\sin(\pi\beta_j-\pi\alpha_j)\prod_{j\neq k}
\frac{\sin(\pi\beta_k-\pi\alpha_j)}{\sin(\pi\alpha_k-\pi\alpha_j)}.
$$
\end{corollary}

\begin{proof}
In the expression for $d_i$ in Theorem \ref{thm-bh} we distribute the square 
root terms so that each bracket becomes purely imaginary. This yields 
$$
d_j= -i\bigl(b_j^{1/2}\bar{a}_j^{1/2}-\bar{b}_j^{1/2}a_j^{1/2}\bigr)
\prod_{k\neq j}
\frac{\bigl(b_k^{1/2}\bar{a}_j^{1/2}-\bar{b}_k^{1/2}a_j^{1/2}\bigr)}
{\bigl(a_k^{1/2}\bar{a}_j^{1/2}-\bar{a}_k^{1/2}a_j^{1/2}\bigr)}.
$$
Then substituting for $a_j^{1/2}=e^{\pi i\alpha_j}$ and $b_j^{1/2}=e^{\pi i\beta_j}$
gives the result.
\end{proof}

\medskip

We now briefly connect our proof of Theorem~\ref{thm-bh} with the one given 
by Beukers and Heckman on page 335 of \cite{BH}. They write $\zeta$ for
a solution of $-1=\zeta^2(b_1b_2b_3)/(a_1a_2a_3)$. 
In Proposition~4.4 of \cite{BH}, they define a vector 
${\bf u}$ satisfying  
$$
\zeta\bigl(B'{A'}^{-1}-I\bigr){\bf z}
=\pm\langle{\bf z},{\bf u}\rangle{\bf u}.
$$
where $A'=U^{-1}AU$ and $B'=U^{-1}BU$ as in \eqref{eq-AB-3}.  
We have
$$
-ie^{-i\psi}\bigl(B'{A'}^{-1}-I\bigr)=
-ie^{-i\psi}\left(\begin{matrix}
c_1 & c_2 & c_3 \\ c_1 & c_2 & c_3 \\ c_1 & c_2 & c_3
\end{matrix}\right)
=\left(\begin{matrix} 1 \\ 1 \\ 1 \end{matrix}\right)
(1,\ 1,\ 1)\left(\begin{matrix}
d_1 & 0 & 0 \\ 0 & d_2 & 0 \\ 0 & 0 & d_3 \end{matrix}\right).
$$
Therefore we may take $\zeta=\mp ie^{i\psi}$ and
$$
{\bf u}=\left(\begin{matrix} 1 \\ 1 \\ 1 \end{matrix}\right).
$$
Following \cite{BH}, we decompose this vector as 
${\bf u}={\bf u}_1+{\bf u}_2+{\bf u}_3$ where
${\bf u}_j$ is an $a_j$-eigenvector of $A'$. Since $A'$ is diagonal,
it is clear that ${\bf u}_j$ is the $j$th vector in the standard
basis for ${\mathbb C}^3$. Since the Hermitian form
is ${\rm diag}(d_1,d_2,d_3)$ this yields
$\langle{\bf u}_j,{\bf u}_j\rangle=d_j$, which agrees with
the statement given on page 335 of \cite{BH}.

\medskip

Beukers and Heckman give a list of angle parameters $\alpha_j$ and $\beta_j$ 
corresponding to finite primitive hypergeometric groups, that is groups where 
the Hermitian form is positive definite; Table 8.3 of \cite{BH}.  
For convenience we reproduce this table below.
All except one of these groups is a Shephard-Todd group \cite{ST} 
and Beukers and Heckman give the Shephard-Todd number (ST in the table below).
(See also \cite{Deraux-CHL} for further connections with Shephard-Todd groups.)
They also give the field generated by the coefficients of the characteristic
polynomials of $A$ and $B$. 
$$
\begin{array}{|r|lll|lll|l|l|}
\hline
{\rm BH} & \alpha_1 & \alpha_2 & \alpha_3 & \beta_1 & \beta_2 & \beta_3 
& \hbox{Field} & \hbox{ST} \\
\hline
2 & 3/14 & 5/14 & 13/14 & 0 & 1/3 & 2/3 
& {\mathbb Q}(i\sqrt{7}) & 24 \\
3 & 3/14 & 5/14 & 13/14 & 0 & 1/4 & 3/4 
& {\mathbb Q}(i\sqrt{7}) & 24 \\
4 & 3/14 & 5/14 & 13/14 & 1/7 & 2/7 & 4/7 
& {\mathbb Q}(i\sqrt{7}) & 24 \\
\hline
5 & 0 & 1/5 & 4/5 & 1/6 & 1/2 & 5/6 
& {\mathbb Q}(\sqrt{5}) & 23 \\
6 & 0 & 1/5 & 4/5 & 1/10 & 1/2 & 9/10 
& {\mathbb Q}(\sqrt{5}) & 23 \\
\hline
7 & 1/6 & 11/30 & 29/30 & 0 & 1/5 & 4/5 
& {\mathbb Q}(\sqrt{5},i\sqrt{3}) & 27 \\
8 & 1/6 & 11/30 & 29/30 & 0 & 1/4 & 3/4 
& {\mathbb Q}(\sqrt{5},i\sqrt{3}) & 27 \\
\hline
9 & 1/6 & 2/3 & 5/6 & 0 & 1/4 & 3/4 
& {\mathbb Q}(i\sqrt{3}) & 25 \\
10 & 1/9 & 4/9 & 7/9 & 0 & 1/6 & 1/2 
& {\mathbb Q}(i\sqrt{3}) & 25 \\
11 & 1/9 & 4/9 & 7/9 & 0 & 1/4 & 3/4 
& {\mathbb Q}(i\sqrt{3}) & 25 \\
\hline
12 & 1/12 & 7/12 & 5/6 & 0 & 1/4 & 3/4 
& {\mathbb Q}(i\sqrt{3}) & -- \\
\hline
\end{array}
$$
For the group on the last line, the reflection group 
$\langle A^k(BA^{-1})A^{-k}\rangle$ is imprimitive
(see Section 5 of Beukers and Heckman), meaning that there is a direct 
sum decomposition of ${\mathbb R}^3$ that is permuted
by $\langle A^k(BA^{-1})A^{-k}\rangle$. Specifically, in this case,
$$
A=\left(\begin{matrix} -\omega & -\bar\omega & -1 \\ 
1 & 0 & 0 \\ 0 & 1 & 0 \end{matrix} \right),\quad
B=\left(\begin{matrix} 
1 & -1 & 1 \\ 1 & 0 & 0 \\ 0 & 1 & 0 \end{matrix} \right).
$$
The complex reflections $A^k(BA^{-1})A^{-k}$ preserve the decomposition
${\mathbb R}^3=V_1\oplus V_2\oplus V_3$ where
$$
V_1={\rm span}\left\{\left(\begin{matrix} 
1 \\ 1 \\ -\omega \end{matrix}\right)\right\},\quad
V_2={\rm span}\left\{\left(\begin{matrix} 
1 \\ \omega \\ -1 \end{matrix}\right)\right\},\quad
V_3={\rm span}\left\{\left(\begin{matrix} 
1 \\ 1 \\ -\bar\omega \end{matrix}\right)\right\}.
$$

In later sections we relate the groups in this table to groups generated by
complex reflections. For ease of reference, we refer to them as
BH2 to BH12.

\section{Groups generated by three complex reflections}
\label{sec-triangle-groups}

In this section we consider subgroups of ${\rm PGL}(3,{\mathbb C})$
generated by three complex reflections. These groups preserve a
Hermitian form. Depending on the signature of this form the group
acts on ${\bf P}^2_{\mathbb C}$, ${\bf E}^2_{\mathbb C}$ or ${\bf H}^2_{\mathbb C}$. 
We are interested when the group is a lattice, meaning that it is
discrete and the quotient of the above space by this group has finite
volume. Our main reference is Deraux, Parker and Paupert \cite{DPP}, 
which builds on several earlier papers including Mostow \cite{Mostow-RCP}, 
Parker and Paupert \cite{PP} and Thompson \cite{Thompson}.

\subsection{Parameters and angles}

Recall that an element $R$ of ${\rm PGL}(3,{\mathbb C})$ is a complex
reflection with angle $\psi$ if $(R-I)$ has rank one and $R$
has determinant $e^{i\psi}$. We consider subgroups of 
${\rm PGL}(3,{\mathbb R})$ generated by three complex reflections, 
each with angle $\psi=2\pi/p$. The space of conjugacy classes of such
groups has four dimensions; see Pratoussevitch \cite{AP} for
example. 

Following Mostow \cite{Mostow-RCP} we normalise the three complex reflections 
$R_1$, $R_2$ and $R_3$ so that the $e^{2\pi i/p}$-eigenspace of $R_j$ is 
spanned by the $j$th standard basis vectors. 
Note that, rather than normalising the determinants to be 1, 
we normalise that $(R_j-I)$ has rank one. 
(See also  Section~2.5 of Parker, Paupert \cite{PP} for a similar 
normalisation in the case of 3-fold symmetry.) Specifically:
\begin{eqnarray*}
R_1 & = & \left(\begin{matrix} 
e^{2\pi i/p} & \rho & -\bar\tau \\ 
0 & 1 & 0 \\ 0 & 0 & 1 \end{matrix}\right), \\
R_2 & = & \left(\begin{matrix}
1 & 0 & 0 \\ -e^{2\pi i/p}\bar\rho & e^{2\pi i/p} & \sigma \\
0 & 0 & 1 \end{matrix}\right), \\
R_3 & = & \left(\begin{matrix}
1 & 0 & 0 \\ 0 & 1 & 0 \\
e^{2\pi i/p}\tau & -e^{2\pi i/p}\bar\sigma & e^{2\pi i/p} \end{matrix}\right).
\end{eqnarray*}
These matrices preserve the Hermitian form (which is $1/2\sin(\pi/p)$ times
the form in \cite{DPP}):
\begin{equation}\label{eq-H}
H=\left(\begin{matrix} 
2\sin(\pi/p) & -ie^{-i\pi/p}\rho & ie^{-i\pi/p}\bar\tau \\
ie^{i\pi/p}\bar\rho & 2\sin(\pi/p) & -ie^{-i\pi/p}\sigma \\
-ie^{i\pi/p}\tau & ie^{i\pi/p}\bar\sigma & 2\sin(\pi/p)
\end{matrix}\right).
\end{equation}
We claimed above that the space of conjugacy classes of triples of complex
reflections all with angle $2\pi/p$ has dimension four. Hence, there
is some redundancy in the above parametrisation by three complex numbers.
Here is a precise statement which combines results found in 
Proposition 3.3 of \cite{DPP}, Section 10 of Pratoussevitch \cite{AP} 
or Section 2.3 of Thompson \cite{Thompson}.

\begin{proposition}\label{prop-parameters}
Let $R_j$, $R'_j$ for $j=1,\,2,\,3$ be complex reflections with angle
$2\pi/p$. Let $\rho,\,\sigma,\,\tau$ and $\rho',\,\sigma',\,\tau'$
be as defined above. The triples $R_1\,R_2,\,R_3$ and 
$R'_1$, $R'_2$, $R'_3$ are conjugate in ${\rm PGL}(3,{\mathbb C})$
if and only if one of the following is true:
\begin{enumerate}
\item[(1)] If $\rho\sigma\tau\neq 0$ and $p\ge 3$ then
$$
|\rho'|=|\rho|,\quad |\sigma'|=|\sigma|,\quad |\tau'|=|\tau|,\quad
\arg(\rho'\sigma'\tau')=\arg(\rho'\sigma'\tau).
$$
\item[(2)] If $\rho\sigma\tau\neq 0$ and $p=2$ then
$$
|\rho'|=|\rho|,\quad |\sigma'|=|\sigma|,\quad |\tau'|=|\tau|,\quad
\arg(\rho'\sigma'\tau')=\pm\arg(\rho'\sigma'\tau).
$$
\item[(3)] If $\rho\sigma\tau=0$ then
$$
|\rho'|=|\rho|,\quad |\sigma'|=|\sigma|,\quad |\tau'|=|\tau|.
$$
\end{enumerate}
\end{proposition}

Note that this result means we can freely choose the arguments of two of 
$\rho$, $\sigma$ and $\tau$. Furthermore, if $\rho\sigma\tau=0$
we may assume, without loss of generality, that $\sigma=0$ and that 
$\rho$, $\tau$ are non-negative real numbers.

Let $n$ be a natural number. We say that $A$ and $B$ satisfy a braid relation 
of length $n$, written ${\rm br}_n(A,B)$, if
$$
(AB)^{n/2}=(BA)^{n/2}.
$$
This notation means the alternating products of $A$ and $B$ with $n$ terms 
starting with $A$ and $B$ respectively. So ${\rm br}_2(A,B)$ is $AB=BA$, that 
is $A$ and $B$ commute, and ${\rm br}_3(A,B)$ is $ABA=BAB$, which is the 
classical braid relation. We write ${\rm br}(A,B)=n$ for the smallest
positive integer for which the braid relation ${\rm br}(A,B)_n$ holds.

Using Pratoussevitch's formulae \cite{AP}, or by direct calculation we find: 
$$
\begin{array}{ll}
{\rm tr}(R_1R_2)=e^{2\pi i/p}(2-|\rho|^2)+1,\quad &
{\rm tr}(R_2R_3)=e^{2\pi i/p}(2-|\sigma|^2)+1,\\
{\rm tr}(R_3R_1)=e^{2\pi i/p}(2-|\tau|^2)+1,\quad &
{\rm tr}(R_1R_3^{-1}R_2R_3)=e^{2\pi i/p}(2-|\sigma\tau-\bar\rho|^2)+1.
\end{array}
$$
Using this, following Proposition~2.3 of \cite{DPP}, we observe that 
\begin{itemize}
\item if $|\rho|=2\cos(\pi/c)$ with $c\in{\mathbb N}$ 
then ${\rm br}(R_1,R_2)=c$;
\item if $|\sigma|=2\cos(\pi/a)$ with $a\in{\mathbb N}$ 
then ${\rm br}(R_2,R_3)=a$;
\item if $|\tau|=2\cos(\pi/b)$ with $b\in{\mathbb N}$ 
then ${\rm br}(R_3,R_1)=b$;
\item if $|\sigma\tau-\bar\rho|=2\cos(\pi/d)$ 
with $d\in{\mathbb N}$ then ${\rm br}(R_1,R_3^{-1}R_2R_3)=d$.
\end{itemize}
In the case above we say the group has {\bf braiding parameters} $(a,b,c;d)$. 
Hence the braiding parameters $(a,b,c;d)$ completely determine 
$|\rho|$, $|\sigma|$, $|\tau|$ and $\Re(\rho\sigma\tau)$. Thus,
together with $p$, they almost determine two possible groups 
$\langle R_1,R_2,R_3\rangle$ up to conjugacy, the ambiguity
coming from the sign of $\Im(\rho\sigma\tau)$; see Remark~3.1 of \cite{DPP}.
However, conjugating by an antiholomorphic map (for example complex
conjugation) has the effect of changing the sign of both $p$ and
$\arg(\rho\sigma\tau)$. Hence we may assume $\arg(\rho\sigma\tau)\in[0,\pi]$
if we allow $p$ to be negative. Thus, it follows from 
Proposition~\ref{prop-parameters} that, if we allow $p$ to be negative,
the braiding parameters $(a,b,c;d)$ determine the group 
$\langle R_1,R_2,R_3\rangle$ up to conjugacy (possibly by an
antiholomorphic map). 

There is a special case when $\sigma=0$ and so $a=2$. This means
that ${\rm br}(R_2,R_3)=2$, that is $R_2$ and $R_3$ commute. Also, 
$\arg(\rho\sigma\tau)$ is not defined, since $\sigma=0$. Moreover, 
we have $2\cos(\pi/d)=|\bar\rho|=2\cos(\pi/c)$ and so $d=c$.
As we observed above, $|\rho|$ and $|\tau|$ completely determine the group
and so, together with $p$, the parameters $b$ and $c$ determine the group,
which has braiding parameters $(2,b,c;c)$.

\medskip

\subsection{Lattices and arithmeticity}

In the tables below we write down the four braiding parameters $(a,b,c;d)$ 
(as described in the previous section) for each of the lattices constructed 
in Deraux-Parker-Paupert \cite{DPP}.
For each of these sets of parameters we give the values of $p$ (which may
be negative) where
the corresponding group is discrete. Depending on the signature of
the Hermitian form $H$ these act on one of ${\bf P}^2_{\mathbb C}$, 
${\bf E}^2_{\mathbb C}$ or ${\bf H}^2_{\mathbb C}$. In the last of these
cases, we write the value of $p$ in bold face when 
$\langle R_1,R_2,R_3\rangle$ is non-arithmetic. The arithmeticity criterion
we use is due to Mostow \cite{Mostow-RCP}. We will not go into details
of how to apply this criterion here, since it has been discussed 
at length for these groups in the paper \cite{DPP}. 
In the next section (Propositions~\ref{prop-levelt-3fold}, 
\ref{prop-levelt-2fold}, \ref{prop-levelt-334}, \ref{prop-levelt-23n}
and \ref{prop-levelt-344} respectively) we will show that all of these 
groups are monodromy groups of higher hypergeometric functions except 
those for $(3,4,4;4)$ with $p$ a multiple of $3$. We indicate these 
values of $p$ in parenthesis.

In the last column we give notes indicating more information about
these groups. If a group is written in square brackets then it indicates
that the two groups are commensurable but that the standard generators
do not correspond to the generators we give. 
When the braiding parameters are $(3,3,3;m)$ then $\langle R_1,R_2,R_3\rangle$
is the Deligne-Mostow group $\Gamma(p,t)$ where $t=1/p+2/m-1/2$ is
Mostow's parameter; see
Section A.9 of \cite{DPP}. The sporadic groups ${\mathcal S}(p,-)$ and 
Thomson groups ${\mathcal T}(p,-)$ are described
in Sections A.1 to A.8 of \cite{DPP}. Commensurability relations between
them are given in Section 7 of \cite{DPP}. The relationship between these
groups and the Couwenberg-Heckman-Looijenga lattices is given in
Theorem 1 of Deraux \cite{Deraux-CHL}. The groups with rotations of order 
$p$ and braiding parameters $(n,n,n;n)$ for $n=4,\,5$ are subgroups of
$\Gamma(n,1/n+2/p-1/2)$, that is lattices with rotations of 
order $n$ and braiding parameters $(3,3,3;p)$; see Proposition 5.1 of 
Parker and Paupert \cite{PP}.
Finally, we indicate which of the Beukers-Heckman groups (acting
on ${\bf P}_{\mathbb C}^2$) lie in this
family; see Propositions~\ref{prop-compare-bh2-4}, \ref{prop-compare-bh-5-6},
\ref{prop-bh-9-10-11}, \ref{prop-compare-bh3}, \ref{prop-compare-bh12} 
and \ref{prop-compare-bh-7-8} below.

\begin{enumerate}
\item[(1)] Three-fold symmetry $(n,n,n;m)$
$$
\begin{array}{|rr|l|l|l|l|}
\hline
 n & m & {\bf P}^2_{\mathbb C} & {\bf E}^2_{\mathbb C} &
 {\bf H}^2_{\mathbb C} & \hbox{Notes} \\
\hline
3 & 2 & 2,\ 3 & 4 & 5,\ 6,\ 7,\ 8,\ {\bf 9},\ 10,\ 12,\ 18 &
\Gamma(p,1/p+1/2),\ {\mathcal C}(p,G_{25}),\ {\rm BH}11 \\
3 & 3 & 2 & & 4,\ 5,\ 6,\ {\bf 7},\ {\bf 8}, 9,\ {\bf 10},\ 12,\ 18 &
\Gamma(p,1/p+1/6),\, [{\rm BH}12] \\
3 & 4 & 2 & & 3,\ 4,\ {\bf 5},\ {\bf 6},\ 8,\ 12 & \Gamma(p,1/p) \\
3 & 5 & 2 & & 3,\ {\bf 4},\ 5,\ 10 & \Gamma(p,1/p-1/10) \\
3 & 6 & 2 & & 3, {\bf 4},\ {\bf 6} & \Gamma(p,1/p-1/6) \\
3 & 7 & 2 & & {\bf 3},\ 7 & \Gamma(p,1/p-3/14) \\
3 & 8 & 2 & & 3,\ 4 & \Gamma(p,1/p-1/4) \\
3 & 9 & 2 & & 3 & \Gamma(p,1/p-5/18) \\
3 & 10 & 2 & & 3 & \Gamma(p,1/p-3/10) \\
3 & 12 & 2 & & 3 & \Gamma(p,1/p-1/3) \\
4 & 3 & 2 & & 3,\ {\bf 4},\ {\bf 5},\ {\bf 6},\ {\bf 8}, {\bf 12} & 
{\mathcal S}(p,\bar\sigma_4),{\mathcal C}(p,G_{24}),\ {\rm BH}2 \\
4 & 4 & & 2 & 3,\ 4,\ {\bf 5},\ {\bf 6},\ 8,\ 12 & 
[\Gamma(4,2/p-1/4)] \\
4 & 5 & & & 2,\ {\bf 3},\ {\bf 4} & {\mathcal S}(p,\sigma_5) \\
5 & 3 & 2 & & 3,\ 4,\ 5,\ 10 & {\mathcal S}(p,\sigma_{10}),\ 
{\mathcal C}(p,G_{23}),\ {\rm BH}5,\, \\
5 & 5 & & & 2,\ 3,\ {\bf 4},\ 5,\ 10 & [\Gamma(5,2/p-3/10)] \\
6 & 4 & & 2 & {\bf 3},\ {\bf 4}, {\bf 6} & {\mathcal S}(p,\sigma_1) \\
\hline
\end{array}
$$
\item[(2)] Two-fold symmetry $(n,n,m;m)$
$$
\begin{array}{|rr|l|l|l|l|}
\hline
n & m & {\bf P}^2_{\mathbb C} & {\bf E}^2_{\mathbb C} &
{\bf H}^2_{\mathbb C} & \hbox{Notes} \\
\hline
3 & 3 & 2 & & 4,\ 5,\ 6,\ {\bf 7},\ {\bf 8}, 9,\ {\bf 10},\ 12,\ 18 &
{\rm BH}12 \\
3 & 4 & 2 & & 3,\ {\bf 4},\ {\bf 5},\ {\bf 6},\ {\bf 8}, {\bf 12} & 
{\mathcal T}(p,{\bf S}_1) \\
3 & 5 & & & 2,\ {\bf 3},\ 5,\ 10,\ -5 & {\mathcal T}(p,{\bf H}_2) \\
4 & 3 & 2 & 3 & 4,\ 5,\ 6,\ 8,\ 12 & [{\mathcal T}(p,{\bf S}_4)] \\
4 & 4 & & 2 & 3,\ 4,\ {\bf 5},\ {\bf 6},\ 8,\ 12 & [\Gamma(4,2/p-1/4)] \\
5 & 4 & 2 & & 3,\ {\bf 4},\ {\bf 5} & [{\mathcal T}(p,{\bf S}_2)] \\
5 & 5 & & & 2,\ 3,\ {\bf 4},\ 5,\ 10 & [\Gamma(5,2/p-3/10)] \\
\hline
\end{array}
$$
\item[(3)] $(3,3,4; n)$
$$
\begin{array}{|r|l|l|l|l|}
\hline
n & {\bf P}^2_{\mathbb C} & {\bf E}^2_{\mathbb C} &
{\bf H}^2_{\mathbb C} & \hbox{Notes} \\
\hline
3 & 2 & & 3,\ 4,\ {\bf 5},\ {\bf 6},\ 8,\ 12 & [\Gamma(p,1/p)] \\
4 & 2 & & 3,\ {\bf 4},\ {\bf 5},\ {\bf 6},\ {\bf 8}, {\bf 12} & 
{\mathcal T}(p,{\bf S}_1),\ [{\mathcal S}(p,\bar\sigma_4)] \\
5 & 2 & & 3,\ {\bf 4},\ {\bf 5} & {\mathcal T}(p,{\bf S}_2),\ 
{\mathcal C}(p,G_{27}),\ {\rm BH}7 \\
6 & & 2 & {\bf 3},\ {\bf 4},\ {\bf 5} & 
{\mathcal T}(p,{\bf E}_1),\ [{\mathcal S}(p,\sigma_1)] \\
7 & & & 2,\ 7 & {\mathcal T}(p,\bar{\bf H}_1) \\
\hline
\end{array}
$$
\item[(4)] $(2,3,n;n)$
$$
\begin{array}{|r|l|l|l|l|}
\hline
 n & {\bf P}^2_{\mathbb C} & {\bf E}^2_{\mathbb C} &
 {\bf H}^2_{\mathbb C} & \hbox{Notes} \\
\hline
3 & 2,\ 3 & 4 & 5,\ 6,\ 7,\ 8,\ {\bf 9},\ 10,\ 12,\ 18 & 
{\mathcal T}(p,{\bf S}_3),\ [\Gamma(p,1/p+1/2)] \\
4 & 2 & 3 & 4,\ 5,\ 6,\ 8,\ 12 & {\mathcal T}(p,{\bf S}_4),\ 
{\mathcal C}(p,G_{26}),\ [\Gamma(p,3/p-1/2)] \\
5 & 2 &  & 3,\ 4,\ 5,\ 10 & {\mathcal T}(p,{\bf S}_5), 
[{\mathcal S}(p,\sigma_{10})] \\
6 & & 2 & 3,\ 4,\ 6 & {\mathcal T}(p,{\bf E}_3)  \\
\hline
\end{array}
$$

\item[(5)] $(3,4,4;4)$ 
$$
\begin{array}{|l|l|l|l|}
\hline
{\bf P}^2_{\mathbb C} & {\bf E}^2_{\mathbb C} &
{\bf H}^2_{\mathbb C} & \hbox{Notes} \\
\hline
& 2 & (3),\ {\bf 4},\ (6),\ (12) & {\mathcal T}(p,{\bf E}_2) \\
\hline
\end{array}
$$
\end{enumerate}

\section{Groups generated by three reflections as higher 
monodromy groups}\label{sec-main-proof}

In this section we prove the main result of the paper. To do so, we
compare the hypergeometric groups described in 
Section~\ref{sec-higher-monodromy} and the groups generated by
three complex reflections described in Section~\ref{sec-triangle-groups},
in particular non-arithmetic lattices in ${\rm PU}(2,1)$ with this property. 
The main issue is that the hypergeometric groups have two generators whereas
the complex reflection groups have three generators. We get around this
in several ways. First, for some groups there is an extra symmetry, 
which means that the group generated by complex reflections has finite
index in a two-generator group. Secondly, for other groups we use known 
relations between the generators to find a generating set with two
elements. In the final subsection, we discuss how this approach goes wrong 
for the groups with braiding parameters $(3,4,4;4)$ and angle $2\pi/p$ where 
$p$ is divisible by 3. 

We consider the groups from \cite{DPP} in families corresponding to the
tables is Section~\ref{sec-triangle-groups}, treating each family in
a separate section. For each family of groups (possibly by adjoining
symmetries) we give generators satisfying 
Levelt's criterion from Theorem~\ref{thm-levelt}, showing
that the reflection group is (commensurable to) a hypergeometric group.
We explicitly write down the parameters $\alpha_j$ and $\beta_j$ of the 
associated hypergeometric equation \eqref{eq-HG-eq} for each group. 
We also give connections of these groups to Beukers and Heckman's
groups BH2 to BH12 \cite{BH}. Finally, for the first two families we show
how to pass between the Hermitian form $D$ given in \cite{BH} and the
form $H$ given in \cite{DPP}. A similar construction applies in the
other cases, but the precise expressions are more complicated.

\subsection{Lattices with three-fold symmetry}

We suppose we have braid relations
$$
{\rm br}(R_1,R_2)={\rm br}(R_2,R_3)={\rm br}(R_3,R_1)=n,\quad 
{\rm br}(R_1,R_3^{-1}R_2R_3)=m.
$$
(Such groups with $n=3$ were studied by Mostow \cite{Mostow-RCP}.)
In this case $\rho=\sigma=\tau$ and there is a symmetry $J$ which
(projectively) has order 3 and satisfies
$R_2=JR_1J^{-1}$ and $R_3=JR_2J^{-1}=J^{-1}R_1J$. The parameter $\tau$ is
determined up to complex conjugation by
$$
|\tau|=2\cos(\pi/n),\quad |\tau^2-\bar\tau|=2\cos(\pi/m).
$$
The particular values of $\tau$ giving rise to lattices are:
$$
\begin{array}{|lll|lll|}
\hline
n & m & \tau & n & m & \tau \\
\hline
3 & k & -e^{-2\pi i/3k} & k & k & e^{4\pi i/3k}+e^{-2\pi i/3k} \\
4 & 3 & (-1-i\sqrt{7})/2 & 4 & 5 & e^{-i\pi/9}(\sqrt{5}+i\sqrt{3})/2 \\
5 & 3 & (1+\sqrt{5})/2 & 6 & 4 & -1+i\sqrt{2} \\
\hline
\end{array}
$$
In the $(3,k)$ case we get $k=2$, $3$, $4$, $5$, $6$, $7$, $8$, $9$,
$10$ or $12$; in the $(k,k)$ case we get a lattice when $k=3$, $4$ or $5$.

The group $\langle R_1,J\rangle$ contains $\langle R_1,R_2,R_3\rangle$
as a subgroup of index 1 or 3. It has generators:
\begin{equation}\label{eq-AB-3fold}
A=J=\left(\begin{matrix} 
0 & 0 & e^{-2\pi i/p} \\ 1 & 0 & 0 \\ 0 & 1 & 0 \end{matrix}\right), \quad
B=R_1J=\left(\begin{matrix}
\tau & -\bar\tau & 1 \\ 1 & 0 & 0 \\ 0 & 1 & 0 \end{matrix}\right).
\end{equation}
Observe that these matrices are in the form \eqref{eq-A-general} and 
\eqref{eq-B-general}. Moreover, $BA^{-1}=R_1$ is a complex reflection. 
Therefore, we have:

\begin{proposition}\label{prop-levelt-3fold}
The matrices $A$ and $B$ given by \eqref{eq-AB-3fold} satisfy Levelt's 
criterion and so $\langle R_1,J\rangle$ is a hypergeometric group.
\end{proposition}

Note that we have ${\rm det}(A)=1$, ${\rm tr}(A)=\tau$, 
${\rm det}(B)=e^{-2\pi i/p}$ and ${\rm tr}(B)=0$. 
From this we can find the eigenvalues $a_j=e^{2\pi i\alpha_j}$
and $b_j=e^{2\pi i\beta_j}$ of $A$ and $B$. The parameters $\alpha_j$
and $\beta_j$ are given, up to a scalar shift, by:
$$
\begin{array}{|ll|lll|lll|}
\hline
n & m & \alpha_1 & \alpha_2 & \alpha_3 & \beta_1 & \beta_2 & \beta_3 \\
\hline
3 & k & 1/3-1/3p & 2/3-1/3p & 1-1/3p & 1/6k & 1/2-1/3k & 1/2+1/6k 
\\
4 & 3 & 1/3-1/3p & 2/3-1/3p & 1-1/3p & 3/7 & 5/7 & 6/7 
\\
4 & 4 & 1/3-1/3p & 2/3-1/3p & 1-1/3p & 1/12 & 1/6 & 3/4 
\\
4 & 5 & 1/3-1/3p & 2/3-1/3p & 1-1/3p & 1/9 & 13/90 & 67/90 
\\
5 & 3 & 1/3-1/3p & 2/3-1/3p & 1-1/3p & 0 & 1/5 & 4/5 
\\
5 & 5 & 1/3-1/3p & 2/3-1/3p & 1-1/3p & 1/10 & 2/15 & 23/30 
\\
6 & 4 & 1/3-1/3p & 2/3-1/3p & 1-1/3p & 1/8 & 3/8 & 1/2 
\\
\hline
\end{array}
$$

Comparing these parameters with those from Beukers and Heckman, we can identify
several of the Beukers-Heckman groups, possibly after a scalar shift and
a change of generators.

\begin{proposition}\label{prop-compare-bh2-4}
The groups BH2 and BH4 are isomorphic to the group $\langle R_1,J\rangle$
in the triangle group with braiding parameters $(4,4,4;3)$ 
and $p=2$; that is ${\mathcal S}(2,\bar\sigma_4)$. Specifically:
\begin{enumerate}
\item[(1)] The group BH2 is a scalar shift of ${\mathcal S}(2,\bar\sigma_4)$
with generators $A=R_1J$ and $B=J$.
\item[(2)] The group BH4 is ${\mathcal S}(2,\bar\sigma_4)$ with
generators $A=J^{-1}R_1J^{-1}$ and $B=(R_1J^{-1})^2$.
\end{enumerate}
\end{proposition}

\begin{proof}
The group ${\mathcal S}(2,\bar\sigma_4)$ has $p=2$ and 
$\tau=(-1-i\sqrt{7})/2=e^{6\pi i/7}+e^{10\pi i/7}+e^{12\pi i/7}$.
\begin{enumerate}
\item[(1)]
The parameters of ${\mathcal S}(2,\bar\sigma_4)$ are
$3/7$, $5/7$, $6/7$; $1/6$, $1/2$, $5/6$. Adding $1/2$ to each of these
(mod 1) and reordering (including swapping the roles of $A$ and $B$) gives
$0$, $1/3$, $2/3$; $3/14$, $5/14$, $13/14$, which are the parameters of 
BH2. This proves (1). 
\item[(2)] In the group ${\mathcal S}(2,\bar\sigma_4)$ write 
$A=J^{-1}R_1J^{-1}$ and $B=(R_1J^{-1})^2$. Note that $BA^{-1}=R_1$.
We calculate
$$
{\rm det}(A)={\rm det}(J^{-1}R_1J^{-1})=-1, \quad
{\rm tr}(A)={\rm tr}(J^{-1}R_1J^{-1})=-\tau. 
$$
Thus the eigenvalues of $J^{-1}R_1J^{-1}$ are $-e^{10\pi i/7}$,
$-e^{12\pi i/7}$, $-e^{6\pi i/7}$, so $\alpha_1=3/14$, $\alpha_2=5/14$,
$\alpha_3=13/14$.
Similarly,
$$
{\rm det}(R_1J^{-1})=1,\quad {\rm tr}(R_1J^{-1})=\bar\tau.
$$
Therefore the eigenvalues of $R_1J^{-1}$ are $e^{2\pi i/7}$, $e^{4\pi i/7}$ and
$e^{8\pi i/7}$. Hence the eigenvalues of $B=(R_1J^{-1})^2$ are also
$e^{2\pi i/7}$, $e^{4\pi i/7}$ and $e^{8\pi i/7}$. Therefore $\beta_1=1/7$,
$\beta_2=2/7$ and $\beta_3=4/7$. This means that the parameters of 
$\langle A,B\rangle$ are the same as the parameters of BH4.

It remains to show that $A$ and $B$ generate $\langle R_1,J\rangle$.
Since $R_1J^{-1}$ has order 7 so $B^3=(R_1J^{-1})^6=(R_1J^{-1})^{-1}$. 
Therefore $R_1=BA^{-1}$ and $J^{-1}=AB^3$
so $\langle A,B\rangle=\langle R_1,J\rangle$.
\end{enumerate}
\end{proof}

\medskip

\begin{proposition}\label{prop-compare-bh-5-6}
The groups BH5 and BH6 are isomorphic to the group $\langle R_1,J\rangle$
in the triangle group with braiding
parameters $(5,5,5;3)$ and $p=2$; that is ${\mathcal S}(2,\sigma_{10})$.
Specifically:
\begin{enumerate}
\item[(1)] The group BH5 is ${\mathcal S}(2,\sigma_{10})$ with generators
$A=R_1J$ and $B=J$.
\item[(2)] The group BH6 is ${\mathcal S}(2,\sigma_{10})$ with
generators $A=J^{-1}R_1J^{-1}$ and $B=(R_1J^{-1})^2$.
\end{enumerate}
\end{proposition}

\begin{proof}
The group ${\mathcal S}(2,\sigma_{10})$ has $p=2$ and 
$\tau=(1+\sqrt{5})/2=1+e^{2\pi i/5}+e^{8\pi i/5}$.
\begin{enumerate}
\item[(1)] The parameters of BH5 are $0$, $1/5$, $4/5$; $1/6$, $1/2$, $5/6$
which are the same as ${\mathcal S}(2,\sigma_{10})$ after swapping the roles
of $A$ and $B$.
\item[(2)] This proof is very similar to the proof of 
Proposition~\ref{prop-compare-bh2-4}~(2). The eigenvalues of 
$A=J^{-1}R_1J^{-1}$ are $-e^{8\pi i/5}$, $-1$ and $-e^{2\pi i/5}$. Hence we
have $\alpha_1=3/10$, $\alpha_2=-1$ and $\alpha_3=7/10$. The eigenvalues
of $R_1J^{-1}$ are $1$, $e^{2\pi i/5}$ and $e^{8\pi i/5}$ and so the
eigenvalues of $B=(R_1J^{-1})^2$ are $1$, $e^{4\pi i/5}$ and $e^{6\pi i/5}$.
So $\beta_1=0$, $\beta_2=2/5$ and $\beta_3=3/5$. Performing a scalar shift by 
1/2 gives the parameters of BH6. Finally, $R_1=BA^{-1}$ and $J^{-1}=AB^2$.
\end{enumerate}
\end{proof}

\medskip

\begin{proposition}\label{prop-bh-9-10-11}
The groups BH9, BH10 and BH11 are isomorphic to the 
group $\langle R_1,J\rangle$ in the triangle group  
with braiding parameters $(3,3,3;2)$ and $p=3$;
that is the Deligne-Mostow group $\Gamma(3,5/6)$. Specifically:
\begin{enumerate}
\item[(1)] The group BH9 is $\Gamma(3,5/6)$ 
with generators by $A=R_1J^{-1}$ and $B=R_1^{-1}J^{-1}$.
\item[(2)] The group BH10 is $\Gamma(3,5/6)$
with generators by $A=J^{-1}$ and $B=R_1J^{-1}$.
\item[(3)] The group BH11 is $\Gamma(3,5/6)$ 
with generators $A=J^{-1}$ and $B=R_1^{-1}J^{-1}$.
\end{enumerate}
\end{proposition}

\begin{proof}
The group $\Gamma(3,5/6)$ has $p=3$ and 
$\tau=-e^{-\pi i/3}=e^{2\pi i/3}=e^{2\pi i/3}+e^{i\pi/6}+e^{7\pi i/6}$. 
\begin{enumerate}
\item[(1)] We take $A=R_1J^{-1}$ and $B=R_1^{-1}J^{-1}$. Note that
$AB^{-1}=(R_1J^{-1})(JR_1)=R_1^2=R_1^{-1}$ since $R_1$ has order 3.
We have 
$$
\begin{array}{ll}
{\rm det}(A)=e^{4\pi i/3},\quad & {\rm tr}(A)=-e^{2\pi i/3}\bar\tau=-1,\\
{\rm det}(B)=1, \quad & {\rm tr}(R_1^{-1}J^{-1})=\bar{\tau}=-e^{i\pi/3}.
\end{array}
$$ 
Therefore the parameters
for this group are $1/3$, $1/2$, $5/6$; $5/12$, $2/3$, $11/12$. These differ
from the parameters for BH9 by a scalar shift of $1/3$.
\item[(2)] We take $A=J^{-1}$ and $B=R_1J^{-1}$. Note that $AB^{-1}=R_1^{-1}$.
We have 
$$
\begin{array}{ll}
{\rm det}(J^{-1})=e^{2\pi i/3}, \quad &{\rm tr}(J^{-1})=0, \\
{\rm det}(R_1J^{-1})=e^{4\pi i/3}, \quad &
{\rm tr}(R_1J^{-1})=-e^{2\pi i/3}\bar\tau=-1.
\end{array}
$$
Therefore, the parameters of this group are 
$1/9$, $4/9$, $7/9$; $1/3$, $1/2$, $5/6$. These differ from the parameters
for the group BH10 by a scalar shift by $2/3$.
\item[(3)] We take $A=J^{-1}$ and $B=R_1^{-1}J^{-1}$. 
Arguing as above, we see the parameters
of this group are $1/9$, $4/9$, $7/9$; $5/12$, $2/3$, $11/12$. These differ
from the parameters of BH11 by a scalar shift of $1/3$.
\end{enumerate}
\end{proof}

\medskip

Finally, we compare the Hermitian forms $D$ from
Theorem~\ref{thm-bh} and $H$ from \eqref{eq-H}. From general theory
(for example the last part of Theorem~\ref{thm-bh}) we know they have 
the same signature, but it is instructive to make this explicit.  We have:
$$
H=\left(\begin{matrix} 
2\sin(\pi/p) & -ie^{-i\pi/p}\tau & ie^{-i\pi/p}\bar\tau \\
ie^{i\pi/p}\bar\tau & 2\sin(\pi/p) & -ie^{-i\pi/p}\tau \\
-ie^{i\pi/p}\tau & ie^{i\pi/p}\bar\tau & 2\sin(\pi/p)
\end{matrix}\right).
$$
Write $\omega=e^{2\pi i/3}$. The eigenvalues of $J$ are
$a_1=\omega e^{-2\pi i/3p}$, $a_2=\bar \omega e^{-2\pi i/3p}$ and 
$a_3=e^{-2\pi i/3p}$. A matrix $V$ whose columns are eigenvectors for $J$ is:
$$
V=\left(\begin{matrix}
\omega e^{-2\pi i/3p}/3 & \bar\omega e^{-2\pi i/3p}/3 & e^{-2\pi i/3p}/3 \\
1/3 & 1/3 & 1/3 \\
\bar\omega e^{2\pi i/3p}/3 & \omega e^{2\pi i/3p}/3 & e^{-2\pi i/3p}/3 
\end{matrix}\right).
$$
Then a short calculation shows that $V^*HV=D={\rm diag}(d_1,\,d_2,\,d_3)$ where
\begin{eqnarray*}
d_1 & = & \bigl(2\sin(\pi/p)-i\bar\omega e^{-i\pi/3p}\tau
+i\omega e^{i\pi/3p}\bar\tau\bigr)/3, \\
d_2 & = & \bigl(2\sin(\pi/p)-i\omega e^{-i\pi/3p}\tau
+i\bar\omega e^{i\pi/3p}\bar\tau\bigr)/3, \\
d_3 & = & \bigl(2\sin(\pi/p)-ie^{-i\pi/3p}\tau+i\bar\tau\bigr)/3.
\end{eqnarray*}
This agrees with the value of $d_j$ found in Theorem~\ref{thm-bh},
namely for $\{j,k,\ell\}=\{1,2,3\}$
$$
d_j=\frac{ie^{-i\pi/p}(a_j^3-\tau a_j^2+\bar\tau a_j-1)}
{a_j(a_k-a_j)(a_\ell-a_j)}.
$$

\subsection{Lattices with two-fold symmetry}

We suppose we have braid relations
$$
{\rm br}(R_2,R_3)={\rm br}(R_3,R_1)=n,\quad 
{\rm br}(R_1,R_2)={\rm br}(R_1,R_3^{-1}R_2R_3)=m.
$$
These groups were studied by Thompson \cite{Thompson} and by
Parker and Sun \cite{PS}. Following \cite{PS}, we choose the normalisation 
$\sigma=\tau=\sqrt{\rho+\overline{\rho}}=2\cos(\pi/n)$ 
and $|\rho|=2\cos(\pi/m)$. There is a symmetry $Q$ with
$$
\begin{array}{ll}
QR_1Q^{-1}=R_1R_2R_1^{-1},\quad &
QR_2Q^{-1}=R_1R_3R_1R_3^{-1}R_1^{-1},\\
QR_3Q^{-1}=R_1R_3R_1^{-1},\quad &
Q^2=R_1R_2R_3.
\end{array}
$$
The values of $\rho$ giving lattices are:
$$
\begin{array}{|lll|lll|}
\hline
n & m & \rho & n & m & \rho \\
\hline
3 & 3 & e^{i\pi /3} & 3 & 4 & (1+i\sqrt{7})/2 \\
3 & 5 & -e^{2\pi i/5}-e^{4i\pi/5} & 4 & 3 & 1 \\
4 & 4 & 1+i & 5 & 4 & e^{i\pi/3}(\sqrt{5}-i\sqrt{3})/2 \\ 
5 & 5 & e^{2\pi i/5}+1 & & & \\
\hline
\end{array}
$$
Note that when $n=m$ we recover groups that also fall into the $(n,n,n;m)$
category, but we obtain a different set of generators.

\begin{lemma}
The group $\langle R_1,R_2,R_3,Q\rangle$ is generated by $R_1$ and $Q$.
\end{lemma}

\begin{proof}
We have
$$
R_2=R_1^{-1}QR_1Q^{-1}R_1,\quad R_3=R_2^{-1}R_1^{-1}Q^2 =(R_1^{-1}Q)^2.
$$
\end{proof}

\medskip

In this case generators are given by:
$$
R_1^{-1}Q=\left(\begin{matrix}
0 & e^{-2\pi i/p} & 0 \\ 1 & 0 & 0 \\ 0 & \sqrt{\rho+\bar\rho} & -1
\end{matrix}\right),\quad
Q=\left(\begin{matrix} \rho & 1-\rho-\bar\rho & \sqrt{\rho+\bar\rho} \\
1 & 0 & 0 \\ 0 & \sqrt{\rho+\bar\rho} & -1 \end{matrix}\right).
$$ 
Conjugating by 
$$
C=\left(\begin{matrix}
1 & 1 & 0 \\ 0 & 1 & 1 \\ 0 & 0 & \sqrt{\rho+\overline{\rho}} 
\end{matrix}\right)
$$
gives $A=C^{-1}(R_1^{-1}Q)C$ and $B=C^{-1}QC$ in the forms
\eqref{eq-A-general} and \eqref{eq-B-general}:
\begin{equation}\label{eq-AB-2fold}
A
=\left(\begin{matrix}
-1 & e^{-2\pi i/p} & e^{-2\pi i/p} \\ 1 & 0 & 0 \\ 0 & 1 & 0 
\end{matrix}\right),\quad
B
=\left(\begin{matrix} 
\rho-1 & 1-\overline{\rho} & 1 \\ 1 & 0 & 0 \\
0 & 1 & 0 \end{matrix}\right).
\end{equation}
Note that $BA^{-1}=C^{-1}R_1C$ is a complex reflection. Therefore, we have:

\begin{proposition}\label{prop-levelt-2fold}
The matrices $A$ and $B$ given \eqref{eq-AB-2fold} satisfy Levelt's 
criterion and so $\langle R_1,R_2,R_3,Q\rangle$ is a hypergeometric group.
\end{proposition}

Note that ${\rm det}(B)=e^{-2\pi i/p}$, ${\rm tr}(B)=-1$, 
${\rm det}(A)=1$ and ${\rm tr}(A)=\rho-1$. From this we can find the 
eigenvalues $a_j=e^{2\pi i\alpha_j}$ and $b_j=e^{2\pi i\beta_j}$ of $A$ and $B$. 
Up to a scalar shift, the parameters $\alpha_j$ and $\beta_j$ are given by:
$$
\begin{array}{|ll|lll|lll|}
\hline
n & m & \alpha_1 & \alpha_2 & \alpha_3 & \beta_1 & \beta_2 & \beta_3 \\
\hline
3 & 3 & 1/2-1/2p & 1/2 & 1-1/2p & 1/12 & 1/3 & 7/12 \\
3 & 4 & 1/2-1/2p & 1/2 & 1-1/2p & 1/7 & 2/7 & 4/7 \\
3 & 5 & 1/2-1/2p & 1/2 & 1-1/2p & 1/15 & 1/5 & 11/15 \\
4 & 3 & 1/2-1/2p & 1/2 & 1-1/2p & 0 & 1/3 & 2/3 \\
4 & 4 & 1/2-1/2p & 1/2 & 1-1/2p & 1/8 & 1/4 & 5/8 \\
5 & 4 & 1/2-1/2p & 1/2 & 1-1/2p & 1/15 & 4/15 & 2/3 \\
5 & 5 & 1/2-1/2p & 1/2 & 1-1/2p & 3/20 & 1/5 & 13/20 \\
\hline
\end{array}
$$

Comparing these parameters with those from Beukers and Heckman, we obtain:

\begin{proposition}\label{prop-compare-bh3}
The group BH3 is isomorphic to the group $\langle R_1,Q\rangle$
in the triangle group with braiding parameters $(3,3,4;4)$ and $p=2$, 
that is the group ${\mathcal T}(2,{\bf S}_1)$, and generators
$A=Q^{-1}$ and $B=R_1Q^{-1}$.
\end{proposition}

\begin{proof}
We have $p=2$ and $\rho-1=(-1+i\sqrt{7})/2=e^{2\pi i/7}+e^{4\pi i/7}+e^{8\pi i/7}$.
Therefore
$$
\begin{array}{ll}
{\rm det}(Q^{-1})=1,\quad & {\rm tr}(Q^{-1})=\bar\rho-1=(-1-i\sqrt{7})/2, \\
{\rm det}(R_1Q^{-1})=-1,\quad & {\rm tr}(R_1Q^{-1})=-1.
\end{array}
$$
Therefore the parameters of this group are $\alpha_1=3/7$, 
$\alpha_2=5/7$, $\alpha_3=6/7$, $\beta_1=1/4$, $\beta_2=1/2$ and 
$\beta_3=3/4$. These differ from the parameters for BH3 by a scalar shift by
$1/2$. 
\end{proof}

\medskip

We remark that Proposition~7.1~(1) of \cite{DPP} says that 
${\mathcal T}(p,{\bf S}_1)$ is isomorphic to ${\mathcal S}(p,\bar\sigma_4)$.
That is, for the same value of $p$ the groups with braiding parameters
$(3,3,4;4)$ and $(4,4,4;3)$ are isomorphic. So this result indicates
BH3 is isomorphic to BH2 and BH4; see \cite{BH}.

\begin{proposition}\label{prop-compare-bh12}
The group BH12 is isomorphic to the group $\langle R_1,Q\rangle$
in the triangle group with braiding parameters $(3,3,3;3)$ and $p=2$, 
that is the Deligne-Mostow group $\Gamma(2,2/3)$, and generators
$A=Q$ and $B=R_1^{-1}Q$.
\end{proposition}

\begin{proof}
We see that the parameters of $\Gamma(2,2/3)$ with these generators are
$\alpha_1=1/12$, $\alpha_2=1/3$, $\alpha_3=7/12$, $\beta_1=1/4$,
$\beta_2=1/2$ and $\beta_3=3/4$. These differ from the parameters of BH12
by a scalar shift of $1/2$.
\end{proof}

\medskip

Finally, we compare the Hermitian forms $H$ and $D$ using the
same method as in the previous section. We have
$$
H=\left(\begin{matrix} 
2\sin(\pi/p) & -ie^{-i\pi/p}\rho & ie^{-i\pi/p}\sqrt{\rho+\bar\rho} \\
ie^{i\pi/p}\bar\rho & 2\sin(\pi/p) & -ie^{-i\pi/p}\sqrt{\rho+\bar\rho} \\
-ie^{i\pi/p}\sqrt{\rho+\bar\rho} & ie^{i\pi/p}\sqrt{\rho+\bar\rho} & 2\sin(\pi/p)
\end{matrix}\right).
$$
The eigenvalues of
$R_1^{-1}Q$ are $a_1=-e^{-i\pi/p}$, $a_2=-1$ and $a_3=e^{-i\pi/p}$ and a matrix
$V$ whose columns are eigenvectors for $R_1^{-1}Q$ are
$$
V=\left(\begin{matrix}
-e^{-i\pi/p}/2 & 0 & e^{-i\pi/p}/2 \\
1/2 & 0 & 1/2 \\
\sqrt{\rho+\bar\rho}/2(1-e^{-i\pi/p}) & \sqrt{\rho+\bar\rho}/2\sin(\pi/p) &
\sqrt{\rho+\bar\rho}/2(1+e^{-i\pi/p})
\end{matrix}\right)
$$
Then we calculate $V^*HV=D={\rm diag}(d_1,\,d_2,\,d_3)$ where
\begin{eqnarray*}
d_1 & = & \frac{4\sin^2(\pi/p)-(1+e^{-i\pi/p})\rho-(1+e^{i\pi/p})\bar\rho}
{4\sin(\pi/p)}, \\
d_2 & = & \frac{\rho+\bar\rho}{2\sin(\pi/p)}, \\
d_3 & = & \frac{4\sin^2(\pi/p)-(1-e^{-i\pi/p})\rho-(1-e^{i\pi/p})\bar\rho}
{4\sin(\pi/p)}.
\end{eqnarray*}
A simple calculation shows that $d_j$ is the value given in 
Theorem~\ref{thm-bh}, namely for $\{j,k,\ell\}=\{1,2,3\}$
$$
d_j=\frac{ie^{-i\pi/p}\bigl(a_j^3-(\rho-1)a_j^2+(\bar\rho-1)a_j-1\bigr)}
{a_j(a_k-a_j)(a_\ell-a_j)}.
$$

\subsection{$(3,3,4;n)$ triangle groups}

In this case we suppose that 
$$
{\rm br}(R_1,R_2)=4, \quad
{\rm br}(R_2,R_3)={\rm br}(R_1,R_3)=3,\quad
{\rm br}(R_1,R_3^{-1}R_2R_3)=n.
$$ 
This means $|\sigma|=|\tau|=1$ and $|\rho|=\sqrt{2}$. Since we are free to 
choose the arguments of two of these parameters, we take $\sigma=\tau=1$;
see Thompson \cite{Thompson}. 
Therefore, $|\rho\sigma-\bar{\tau}|=|\rho-1|=2\cos(\pi/n)$. 
The values of $\rho$ corresponding to lattices are
$$
\begin{array}{|ll|ll|}
\hline 
n & \rho & n & \rho \\
\hline
3 & 1+i &
4 & (1+i\sqrt{7})/2 \\ 
5 & e^{2\pi i/3}(\sqrt{5}-i\sqrt{3})/2 &
6 & i\sqrt{2} \\ 
7 & e^{6\pi i/7}(-1+i\sqrt{7})/2 &
& \\
\hline
\end{array}
$$

\begin{proposition}
Suppose that ${\rm br}(R_1,R_2)=4$ and 
${\rm br}(R_2,R_3)={\rm br}(R_1,R_3)=3$ 
then $\langle R_1,R_2,R_3\rangle$
is generated by $R_2$ and $R_1R_2R_3$.
\end{proposition}

\begin{proof}
Using ${\rm br}(R_2,R_3)=3$, then ${\rm br}(R_1,R_3)=3$ and then
${\rm br}(R_1,R_2)=4$, we have:
\begin{eqnarray*}
\lefteqn{R_2(R_1R_2R_3)^2R_2(R_1R_2R_3)^{-2}R_2^{-1}}\\
& = & 
R_2R_1R_2R_3R_1(R_2R_3R_2R_3^{-1}R_2^{-1})R_1^{-1}R_3^{-1}R_2^{-1}R_1^{-1}R_2^{-1} \\
& = & R_2R_1R_2(R_3R_1R_3R_1^{-1}R_3^{-1})R_2^{-1}R_1^{-1}R_2^{-1} \\
& = & R_2R_1R_2R_1R_2^{-1}R_1^{-1}R_2^{-1} \\
& = & R_1.
\end{eqnarray*}
Then $R_3=R_2^{-1}R_1^{-1}(R_1R_2R_3)$. That is,
$$
(R_1R_2R_3)^2R_2^{-1}(R_1R_2R_3)^{-2}R_2^{-1}(R_1R_2R_3)=R_3.
$$
Therefore $R_1$, $R_2$ and $R_3$ are all contained in 
$\langle R_2,\,(R_1R_2R_3)\rangle$.
\end{proof}

\medskip

This means that $\langle R_1,R_2,R_3\rangle$ is generated by 
$(R_1R_2R_3)$ and $(R_2^{-1}R_1R_2R_3)$. Moreover, 
$(R_1R_2R_3)(R_2^{-1}R_1R_2R_3)^{-1}=R_2$ which is a complex reflection.
We multiply both of these matrices by $e^{2\pi i/p}$, which is a scalar
shift.
\begin{eqnarray*}
e^{-2\pi i/p}R_2^{-1}R_1R_2R_3 & = & \left(\begin{matrix}
\rho-2 & 1 & \rho-1 \\
2-2\bar{\rho}-e^{-2\pi i/p}\bar\rho & \bar\rho+e^{-2\pi i/p} & 2-\bar\rho \\ 
1 & -1 & 1 \end{matrix}\right), \\
e^{-2\pi i/p}R_1R_2R_3 & = & \left(\begin{matrix}
\rho-2 & 1 & \rho-1 \\
1-\bar\rho & 0 & 1 \\
1 & -1 & 1 \end{matrix}\right).
\end{eqnarray*}

Conjugating by
$$
C=\left(\begin{matrix} 0 & 1 & -\rho \\ 1 & 1-\rho & -1 \\
0 & -1 & \rho-1 \end{matrix}\right)
$$
gives matrices in the form \eqref{eq-A-general} and \eqref{eq-B-general}:
\begin{eqnarray}
A & = & C^{-1}(e^{-2\pi i/p}R_2^{-1}R_1R_2R_3)C \notag \\
& = & \left(\begin{matrix}
\rho+\bar\rho-1+e^{-2\pi i/p} & 
-1+e^{-2\pi i/p}(1-\rho-\bar\rho) & e^{-2\pi i/3p} \\
1 & 0 & 0 \\ 0 & 1 & 0 \end{matrix}\right), 
\label{eq-A-334} \\
B & = & C^{-1}(e^{-2\pi i/p}R_1R_2R_3)C \notag \\
& = & \left(\begin{matrix}
\rho-1 & 1-\bar\rho & 1 \\ 1 & 0 & 0 \\ 0 & 1 & 0 \end{matrix}\right), 
\label{eq-B-334} \\
\end{eqnarray}
Observe that $BA^{-1}=C^{-1}R_2C$ is a complex reflection. Therefore, we have:

\begin{proposition}\label{prop-levelt-334}
The matrices $A$ and $B$ given by \eqref{eq-A-334} and \eqref{eq-B-334} 
satisfy Levelt's criterion and so $\langle R_1,R_2,R_3\rangle$ is a 
hypergeometric group.
\end{proposition}

Therefore, we have ${\rm det}(B)=1$, ${\rm tr}(B)=\rho-1$,
${\rm det}(A)=e^{-2\pi i/p}$ and 
$$
{\rm tr}(A)=\rho+\overline{\rho}-1+e^{-2\pi i/p}
=2-|\rho-1|^2+e^{-2\pi i/p}=-2\cos(2\pi/n)+e^{-2\pi i/p}.
$$
From this we can calculate $a_j=e^{2\pi i\alpha_j}$ and
$b_j=e^{2\pi i\beta_j}$ the eigenvalues of $A$ and $B$.
We have:
$$
\begin{array}{|l|lll|lll|}
\hline
n & \alpha_1 & \alpha_2 & \alpha_3 & \beta_1 & \beta_2 & \beta_3 \\
\hline
3 &  1/2-1/n & 1/2+1/n & 1-1/p & 1/8 & 1/4 & 5/8 \\
4 & 1/2-1/n & 1/2+1/n & 1-1/p & 1/7 & 2/7 & 4/7 \\
5 & 1/2-1/n & 1/2+1/n & 1-1/p & 2/15 & 1/3 & 8/15 \\
6 & 1/2-1/n & 1/2+1/n  & 1-1/p & 1/8 & 3/8 & 1/2 \\
7 & 1/2-1/n & 1/2+1/n & 1-1/p & 23/42 & 4/7 & 37/42 \\
\hline
\end{array}
$$

\begin{proposition}\label{prop-compare-bh-7-8}
The groups BH7 and BH8 are isomorphic to the group with braiding 
parameters $(3,3,4;5)$ and $p=2$; that is ${\mathcal T}(2,{\bf S}_2)$.
Specifically:
\begin{enumerate}
\item[(1)] The group BH7 is ${\mathcal T}(2,{\bf S}_2)$ with generators
$A=R_3R_2R_1$ and $B=R_2R_3R_2R_1$.
\item[(2)] The group BH8 is ${\mathcal T}(2,{\bf S}_2)$ with generators
$A=R_3R_2R_1$ and $B=R_2R_1$.
\end{enumerate}
\end{proposition}

\begin{proof}
\begin{enumerate}
\item[(1)] Let $A=R_3R_2R_1=(R_1R_2R_3)^{-1}$ and
$B=R_2R_3R_2R_1=(R_1R_2R_3R_2^{-1})^{-1}$. We find that
$$
\begin{array}{ll}
{\rm det}(A)=-1,\quad & 
{\rm tr}(A)=1-\bar\rho=e^{i\pi/3}(1+\sqrt{5})/2
=e^{i\pi/3}(1+e^{2\pi i/5}+e^{8\pi i/5}), \\
{\rm det}(B)=1,\quad & 
{\rm tr}(B)=2-\rho-\bar\rho=(1+\sqrt{5})/2=1+e^{2\pi i/5}+e^{8\pi i/5}.
\end{array}
$$
Therefore the parameters of this group are
$\alpha_1=1/6$, $\alpha_2=11/30$, $\alpha_3=29/30$, $\beta_1=0$,
$\beta_2=1/5$ and $\beta_3=4/5$. These are the parameters of 
BH7.
\item[(2)] We take $A=R_3R_2R_1$ and $B=R_2R_1$. The parameters 
$\alpha_j$ are the same as in part (1). We have
$$
{\rm det}(B)=1,\quad {\rm tr}(B)=1=1+i-i.
$$
Therefore the parameters of $B$ are $\beta_1=0$, $\beta_2=1/4$ and
$\beta_3=3/4$. These are the parameters of BH8.
\end{enumerate}
\end{proof}

\medskip

We could compare the Hermitian forms $D$ and $H$ as in earlier sections.
The same method works, but the matrix $V$ is slightly harder to write down.
Therefore we leave the details to the reader.

\subsection{$(2,3,n;n)$ triangle groups}

In this case we suppose
$$
{\rm br}(R_2,R_3)=2,\quad {\rm br}(R_1,R_3)=3,\quad 
{\rm br}(R_1,R_2)=n.
$$ 
This means that $\sigma=0$, $|\tau|=1$ and
$|\rho|=2\cos(\pi/n)$. This group is rigid and we may take $\tau=1$
and $\rho=2\cos(\pi/n)$ for $n=3,\,4,\,5$ or $6$; see \cite{DPP}.

\begin{proposition}
Suppose that ${\rm br}(R_2,R_3)=2$ and ${\rm br}(R_1,R_3)=3$.
Then $\langle R_1,R_2,R_3\rangle$ is generated by $R_3$ and $R_1R_2$.
\end{proposition}

\begin{proof}
Using ${\rm br}(R_2,R_3)=2$ and ${\rm br}(R_1,R_3)=3$ we have
\begin{eqnarray*}
R_1 & = & R_3R_1R_3R_1^{-1}R_3^{-1} \ =\ R_3(R_1R_2)R_3(R_1R_2)^{-1}R_3^{-1}, \\
R_2 & = & R_1^{-1}(R_1R_2) 
\ = \ R_3(R_1R_2)R_3^{-1}(R_1R_2)^{-1}R_3^{-1}(R_1R_2).
\end{eqnarray*}
\end{proof}

\medskip

We take $R_1R_2$ and $R_3R_1R_2$ as generators. We do a scalar
shift by multiplying both generators by $e^{-2\pi i/p}$. We have
\begin{eqnarray*}
e^{-2\pi i/p}R_1R_2 & = & \left(\begin{matrix}
1-|\rho|^2 & \rho & -e^{-2\pi i/p} \\
-\bar\rho & 1 & 0 \\
0 & 0 & e^{-2\pi i/p}
\end{matrix}\right), \\
e^{-2\pi i/p}R_3R_1R_2 & = & \left(\begin{matrix}
1-|\rho|^2 & \rho & -e^{-2\pi i/p} \\ -\bar\rho & 1 & 0 \\ 
e^{2\pi i/p}-e^{2\pi i/p}|\rho|^2 & e^{2\pi i/p}\rho & 0 
\end{matrix}\right).
\end{eqnarray*}
Conjugating by
$$
C=\left(\begin{matrix}
0 & -e^{-2\pi i/p} & e^{-2\pi i/p} \\
0 & 0 & e^{-2\pi i/p}\bar\rho \\
1 & |\rho|^2-2 & 1 \end{matrix}\right)
$$
gives
\begin{eqnarray}
A & = & C^{-1}(e^{-2\pi i/p}R_1R_2)C \notag \\
& = & \left(\begin{matrix}
2-|\rho|^2+e^{-2\pi i/p} & e^{-2\pi i/p}|\rho|^2-2e^{-2\pi i/p}-1 & e^{-2\pi i/p} \\
1 & 0 & 0 \\ 0 & 1 & 0 \end{matrix}\right),
\label{eq-A-23n} \\
B & = & C^{-1}(e^{-2\pi i/p}R_3R_1R_2)C \notag \\
& = & \left(\begin{matrix}
2-|\rho|^2 & |\rho|^2-2 & 1 \\ 1 & 0 & 0 \\ 0 & 1 & 0 \end{matrix}\right), 
\label{eq-B-23n}.
\end{eqnarray}
Observe that $BA^{-1}=C^{-1}R_3C$ is a complex reflection. 
Therefore, we have:

\begin{proposition}\label{prop-levelt-23n}
The matrices $A$ and $B$ given by \eqref{eq-A-23n} and \eqref{eq-B-23n} 
satisfy Levelt's criterion and so $\langle R_1,R_2,R_3\rangle$ is a 
hypergeometric group.
\end{proposition}

Since $|\rho|=2\cos(\pi/n)$ we see that 
$$
{\rm tr}(A)=e^{-2\pi i/p}-2\cos(2\pi/n), \quad 
{\rm tr}(B)=2-4\cos^2(\pi/n)=-2\cos(2\pi/n).
$$
Thus
$$
\begin{array}{|l|lll|lll|}
\hline
n & \alpha_1 & \alpha_2 & \alpha_3 & \beta_1 & \beta_2 & \beta_3 \\
\hline
3 & 1/2-1/n & 1/2+1/n & 1-1/p & 0 & 1/4 & 3/4 
\\
4 & 1/2-1/n & 1/2+1/n & 1-1/p & 0 & 1/3 & 2/3 
\\
5 & 1/2-1/n & 1/2+1/n & 1-1/p & 0 & 2/5 & 3/5 
\\
6 & 1/2-1/n & 1/2+1/n & 1-1/p & 0 & 1/2 & 1/2 
\\
\hline
\end{array}
$$

\bigskip

\subsection{$(3,4,4;4)$ triangle groups}

In this case we suppose we have braid relations
$$
{\rm br}(R_2,R_3)=3,\quad 
{\rm br}(R_1,R_3)={\rm br}(R_1,R_2)={\rm br}(R_1,R_3^{-1}R_2R_3)=4.
$$
Following Thompson \cite{Thompson}, we choose $\rho=\tau=\sqrt{2}$ 
and $\sigma=-\overline{\omega}=e^{i\pi/3}$. 
The group generated by the $R_j$ is called ${\mathcal T}(p,{\bf E}_2)$.
Among the values of $p$ for which these groups are discrete there is one 
non-arithmetic lattice, ${\mathcal T}(4,{\bf E}_2)$. We show below that
this group is a subgroup of a hypergeometric monodromy group. 
Unfortunately, it seems that this is not true for some other values of $p$.

There is a symmetry $S$ of order 3 satisfying:
\begin{equation}\label{eq-E2-S-rels}
SR_1S^{-1}=R_1,\quad
SR_2S^{-1}=R_3,\quad
SR_3S^{-1}=R_3^{-1}R_2R_3,\quad
SR_3^{-1}R_2R_3S^{-1}=R_2.
\end{equation}
In particular, $SR_2R_3S^{-1}=R_2R_3$. 

\begin{proposition}
If $p$ is not divisible by 3 then $\langle R_1,R_2,R_3,S\rangle$ is generated
by $R_3$ and $SR_1$.
\end{proposition}

\begin{proof}
When $p$, which is the order of $R_1$, is not a multiple of 3,
say $p=3m\pm1$, since $S$ and $R_1$ commute and $S^3$ is the identity, 
we have
$$
(SR_1)^{3m}=(S^3)^mR_1^{3m}=R_1^{p\mp 1}=R_1^{\mp 1}.
$$
Thus $R_1$ (and hence $S$) lies in the group $\langle SR_1,\, R_3\rangle$.
Furthermore $R_2=S^{-1}R_3S$ and so $R_1$, $R_2$, $R_3$ and $S$ all
lie in $\langle SR_1,\, R_3\rangle$.
\end{proof}

\medskip

As a matrix $S$ is given by
$$
S=\left(\begin{matrix} 1 & 0 & 0 \\ 0 & 0 & \omega \\
0 & -1 & -\bar\omega \end{matrix}\right).
$$
We take $SR_1$ and $R_3SR_1$ as generators and we perform a scalar shift
by multiplying by $e^{-2\pi i/p}$.
We have
\begin{eqnarray*}
e^{-2\pi i/p}SR_1 & = & \left(\begin{matrix}
1 & \sqrt{2}e^{-2\pi i/p} & -\sqrt{2}e^{-2\pi i/p} \\
0 & 0 & \omega e^{-2\pi i/p}\\
0 & -e^{-2\pi i/p} & -\bar\omega e^{-2\pi i/p} \end{matrix}\right), \\
e^{-2\pi i/p}R_3SR_1 & = & \left(\begin{matrix}
1 & \sqrt{2}e^{-2\pi i/p} & -\sqrt{2}e^{-2\pi i/p} \\
0 & 0 & \omega e^{-2\pi i/p} \\ 
\sqrt{2}e^{2\pi i/p} & 1 & -2 \end{matrix}\right).
\end{eqnarray*}
Conjugating by
$$
C=\left(\begin{matrix}
0 & -\sqrt{2} & \sqrt{2}\omega e^{-2\pi i/p} \\
0 & \omega & -\omega \\
e^{2\pi i/p} & -e^{2\pi i/p} & 0 \end{matrix}\right)
$$
gives
\begin{eqnarray}
A & = & C^{-1}(e^{-2\pi i/p}SR_1)C
=\left(\begin{matrix}
1-\bar\omega e^{-2\pi i/p} & \bar\omega e^{-2\pi i/p}-\omega e^{-4\pi i/p} &
\omega e^{-4\pi i/p} \\
1 & 0 & 0 \\ 0 & 1 & 0 \end{matrix}\right),
\label{eq-A-344} \\
B & = & C^{-1}(e^{-2\pi i/p}R_3SR_1)C
=\left(\begin{matrix} -1 & \omega e^{-2\pi i/p} & \omega e^{-2\pi i/p} \\
1 & 0 & 0 \\ 0 & 1 & 0 \end{matrix}\right).
\label{eq-B-344}
\end{eqnarray}
Note that $BA^{-1}=C^{-1}R_3C$ is a complex reflection. Therefore
we have:

\begin{proposition}\label{prop-levelt-344}
The matrices $A$ and $B$ given by \eqref{eq-A-344} and \eqref{eq-B-344}
satisfy Levelt's criterion. Therefore, when $p$ is not divisible by 3
the group $\langle R_1,R_2,R_3,S\rangle$ is a hypergeometric group.
\end{proposition}

Since we have 
$$
\begin{array}{ll}
{\rm det}(A)=\omega e^{-4\pi i/p},\quad & 
{\rm tr}(A)=1-\bar\omega e^{-2\pi i/p}=1+e^{-2\pi i/p}+\omega e^{-2\pi i/p}, \\
{\rm det}(B)=\omega e^{-2\pi i/p},\quad &
{\rm tr}(B)=-1,
\end{array}
$$
we can calculate the eigenvalues of $A$ as $a_1=1$, $a_2=\omega e^{-2\pi i/p}$,
$a_3=e^{-2\pi i/p}$ and the eigenvalues of $B$ as 
$b_1=-\bar\omega e^{-i\pi/p}$, $b_2=-1$, $b_3=\bar\omega e^{-i\pi /p}$.
Therefore, the group has the following angle parameters:
$$
\begin{array}{|lll|lll|}
\hline
\alpha_1 & \alpha_2 & \alpha_3 & \beta_1 & \beta_2 & \beta_3 \\
\hline
0 & 1/3-1/p & 1-1/p & 1/6-1/2p & 1/2 & 2/3-1/2p 
\\
\hline
\end{array}
$$

\medskip

In contrast, we now show this method will not work for 
$\langle R_1,R_2,R_3,S\rangle$ when $p=3$ or $6$. Of course, this does not rule
out the possibility that these groups may be commensurable to a 
two generator group.

\begin{proposition}
Suppose that $p=3$ or $p=6$. Then the group 
$\Gamma=\langle R_1,R_2,R_3,S\rangle$
does not have a presentation with two generators. In particular, it is not
a hypergeometric group.
\end{proposition}

\begin{proof}
When $p=3,4,6$ a presentation for $\langle R_1,R_2,R_3\rangle$ 
is given in Section A.6 of \cite{DPP}. It is:
$$
\left\langle R_1,\,R_2,\,R_3\ \Big\vert\ \begin{array}{c}
R_1^p,\, R_2^p,\, R_3^p,\,(R_1R_2R_3)^6,\,{\rm br}_3(R_2,R_3),\, \\
{\rm br}_4(R_3,R_1),\,{\rm br}_4(R_1,R_2),\,
{\rm br}_4(R_1,R_2R_3R_2^{-1}),\, {\rm br}_6(R_3,R_1R_2R_1^{-1}),\, \\
(R_1R_2)^{\frac{4p}{p-4}},\, (R_1R_3)^{\frac{4p}{p-4}},\, 
(R_1R_2R_3R_2^{-1})^{\frac{4p}{p-4}},\, (R_3R_1R_2R_1^{-1})^{\frac{3p}{p-3}}
\end{array}\right\rangle.
$$
(A relation should be omitted when the denominator of its exponent
is zero or negative.) 

We now adjoin $S$ and use the relations \eqref{eq-E2-S-rels} to get a
presentation for $\Gamma=\langle R_1,R_2,R_3,S\rangle$. We eliminate the
generator $R_3$ by substituting $R_3=SR_2S^{-1}$.
Note that 
$$
R_2=S(R_3^{-1}R_2R_3)S^{-1}=S^{-1}R_2^{-1}S^{-1}R_2SR_2S
$$ 
is equivalent to ${\rm br}_4(R_2,S)$. In turn, this implies 
\begin{eqnarray*}
R_3^{-1}R_2^{-1}R_3^{-1}R_2R_3R_2 
& = & (SR_2^{-1}S^{-1})R_2^{-1}(SR_2^{-1}S^{-1})R_2(SR_2S^{-1})R_2 \\
& = & S(R_2^{-1}S^{-1}R_2^{-1}S^{-1})S^{-1}R_2^{-1}S^{-1}(R_2SR_2S)SR_2 \\
& = & S(S^{-1}R_2^{-1}S^{-1}R_2^{-1})S^{-1}R_2^{-1}S^{-1}(SR_2SR_2)SR_2 \\
& = & 1.
\end{eqnarray*}
Thus ${\rm br}_4(R_2,S)$ implies 
${\rm br}_3(R_2,R_3)={\rm br}_3(R_2,SR_2S^{-1})$. Hence $\Gamma$ has a
presentation
$$
\left\langle R_1,\,R_2,\,S\ \Big\vert\ \begin{array}{c}
R_1^p,\, R_2^p,\, S^3,\,(R_1R_2SR_2S^{-1})^6,\, \\
{\rm br}_2(R_1,S),\, {\rm br}_4(R_2,S),\,{\rm br}_4(R_1,R_2),\,
{\rm br}_6(SR_2S^{-1},R_1R_2R_1^{-1}),\, \\
(R_1R_2)^{\frac{4p}{p-4}},\, (SR_2S^{-1}R_1R_2R_1^{-1})^{\frac{3p}{p-3}}
\end{array}\right\rangle.
$$

Now consider the abelianisation $\Gamma'$ of $\Gamma$. We claim that
$\Gamma'$ is a direct product of two cyclic groups of order $p$
and a group of order $3$. Since $p=3$ or $6$ is divisible by 3 we see that this
group requires at least three generators. If $\Gamma$ were to have a two
generator presentation then this would lead to a presentation for $\Gamma'$
with at most two generators, which is a contradiction. 
Therefore the result follows from this claim.

We now prove the claim. We investigate the effect of abelianisation on each of
the relations. Recall we are only considering $p=3$ or $p=6$:
\begin{itemize}
\item Since $p$ divides 6 the abelianisation of 
$(R_1R_2SR_2S^{-1})^6$ follows from $R_1^p$ and $R_2^p$.
\item The abelianisation of any braid relation of
even length is the trivial relation.
\item Since $p$ divides $4p/(p-4)$ the abelianisation of 
$(R_1R_2)^{\frac{4p}{p-4}}$ follows from $R_1^p$ and $R_2^p$.
\item When $p=6$, we see that $p$ divides $3p/(p-3)$ and so
the abelianisation of the relation 
$(SR_2S^{-1}R_1R_2R_1^{-1})^{\frac{3p}{p-3}}$ follows from $R_2^p$.
\end{itemize}
Hence, the only relations in $\Gamma'$ are $R_1^p$, $R_2^p$, $S^3$ and
that the generators commute.
Therefore $\Gamma'=\langle R_1\rangle\times\langle R_2\rangle\times
\langle S\rangle$ as claimed.
\end{proof}

\end{document}